\documentclass[12pt]{amsart}

\usepackage[colorlinks=true,linkcolor=blue]{hyperref}
\usepackage{amsmath}
\usepackage{amssymb}
\usepackage{amsthm}
\usepackage{bm} 

\usepackage{graphicx}
\usepackage{epstopdf}
\usepackage{epsfig}

\usepackage{tikz}
\usepackage{tikz-cd}

\theoremstyle{plain} 
\newtheorem{thm}{Theorem}[section]
\newtheorem*{mthm}{Main Theorem}
\newtheorem{prop}[thm]{Proposition}
\newtheorem{lemma}[thm]{Lemma}
\newtheorem{cor}[thm]{Corollary} 

\newtheorem{conj}{Conjecture}

\theoremstyle{remark}
\newtheorem{remark}[thm]{Remark}

\theoremstyle{definition}
\newtheorem{defin}[thm]{Definition}

\newcommand{\eps}{\varepsilon}

\newcommand{\CC}{\mathbb{C}}

\newcommand{\PP}{\mathbb{P}}
\newcommand{\QQ}{\mathbb{Q}}

\newcommand{\ZZ}{\mathbb{Z}}

\newcommand{\Qbar}{\overline{\QQ}}

\newcommand{\Kbar}{\overline{K}}
\newcommand{\kbar}{\overline{k}}

\DeclareMathOperator{\sgn}{sgn}
\DeclareMathOperator{\Par}{Par}

\DeclareMathOperator{\Orb}{Orb}

\DeclareMathOperator{\Gal}{Gal}

\DeclareMathOperator{\Aut}{Aut}

\newcommand{\mapa}{\bm{\alpha}}
\newcommand{\mapb}{\bm{\beta}}
\newcommand{\mape}{\bm{\varepsilon}}
\newcommand{\mapt}{\bm{\theta}}

\newcommand{\la}{\langle}
\newcommand{\ra}{\rangle}

\newcommand{\dsps}{\displaystyle}

\begin{document}

\title{The arithmetic basilica: a quadratic PCF arboreal Galois group}
\date{January 7, 2020; revised September 14, 2021}
\address{Amherst College \\ Amherst, MA 01002 \\ USA}
\author[Ahmad]{Faseeh Ahmad}
\email{faseehirfan@gmail.com}
\author[Benedetto]{Robert L. Benedetto}
\email{rlbenedetto@amherst.edu}
\author[Cain]{Jennifer Cain}
\email{jentcain@gmail.com}
\author[Carroll]{Gregory Carroll}
\email{gcarroll19@amherst.edu}
\author[Fang]{Lily Fang}
\email{lilycfang@yahoo.com}

\begin{abstract}
The arboreal Galois group of a polynomial $f$ over a field $K$ encodes
the action of Galois on the iterated preimages of a root point $x_0\in K$,
analogous to the action of Galois on the $\ell$-power torsion of an abelian variety.
We compute the arboreal Galois group of the postcritically finite polynomial $f(z) = z^2 - 1$
when the field $K$ and root point $x_0$ satisfy a simple condition.
We call the resulting group the \emph{arithmetic basilica group}
because of its relation to the basilica group associated with
the complex dynamics of $f$.
For $K=\QQ$, our condition holds for infinitely many choices of $x_0$.
\end{abstract}

\maketitle

Let $K$ be a field with algebraic closure $\Kbar$, let $x_0\in K$,
and let $f\in K[z]$ be a polynomial of degree $d\geq 2$.
For each $n\geq 0$, let $f^n$ denote
the $n$-th iterate $f\circ \cdots \circ f$ of $f$ under composition,
with $f^0(z)=z$ and $f^1(z)=f(z)$. The \emph{backward orbit} of $x_0$ under $f$ is
\[ \Orb_f^{-}(x_0) := \coprod_{n\geq 0} f^{-n}(x_0) \subseteq \Kbar, \]
where $f^{-n}(y)$ is the set of roots of the equation $f^n(z)=y$ in $\Kbar$.

If $f^n(z)-x_0$ is a separable polynomial, then $f^{-n}(x_0)$ has exactly $d^n$ elements,
and the field $K_n:=K(f^{-n}(x_0))\subseteq \Kbar$ is a Galois extension of $K$, with Galois group
\[ G_n:=\Gal(K_n/K). \]
If $f^n(z)-x_0$ is separable for all $n\geq 0$, then we also define
\[ G_{\infty}:=\Gal(K_{\infty}/K), \quad\text{where}\quad  K_{\infty}:=\bigcup_{n\geq 0} K_n .\]
The backward orbit $\Orb_f^-(x_0)$ has the structure of
an infinite $d$-ary rooted tree $T_{d,\infty}$, formed
by connecting each $y\in f^{-(n+1)}(x_0)$ to $f(y)\in f^{-n}(x_0)$ via an edge.
Thus, $G_\infty$ is isomorphic to a subgroup of $\Aut(T_{d,\infty})$,
and $G_n$ is isomorphic to a subgroup of the automorphism group $\Aut(T_{d,n})$,
where $T_{d,n}$ is the subtree of just the bottom $n$ levels of $T_{d,\infty}$.
The resulting action of Galois on the tree is analogous
to the action of Galois on the $\ell$-power torsion of an abelian variety $A$,
since the $\ell$-power torsion is precisely
the backward orbit of the identity point $O$ under the morphism $[\ell]:A\to A$.

Odoni introduced the study of such Galois groups in 1985 in \cite{Odoni},
and Boston and Jones in 2007 called them ``arboreal'' in \cite{BosJon},
since they act on trees.
These groups have attracted increasing attention over the years;
see \cite{AHM,ABetal,BHL,FM,FPC,GoTa,Hindes,Hindes2,Ing,JonMan,JKMT,Stoll,Swam}
for a (very limited) selection.
See also \cite{Jones} for a survey of the field.
Many examples have been found where $G_{\infty}$ is the full group $\Aut(T_{d,\infty})$,
as in \cite{BenJuu,Juul,Kadets,Looper,Odoni,Specter,Stoll}.
More generally, the expectation has emerged that
when $K$ is a global field, $G_{\infty}$ should usually have finite index in $\Aut(T_{d,\infty})$;
see \cite[Conjecture~3.11]{Jones} for a precise conjecture when $d=2$,
and \cite{BriTuc,GNT,JKetal} for conditional results for $d=2,3$.
By analogy, Serre's Open Image Theorem \cite{Serre} states that
for a non-CM elliptic curve over a number field,
the action of Galois on the $\ell$-power torsion
has finite index in the appropriate automorphism group $GL(2,\ZZ_{\ell})$.

However, just as Serre's Theorem excludes the special case of CM elliptic curves,
there are special situations where $G_{\infty}$ necessarily has infinite index
in $\Aut(T_{d,\infty})$. One such case is that the map $f$ is
\emph{postcritically finite}, or PCF, meaning that for every ramification point $c$
of $f$, the forward orbit $\{f^n(c) | n\geq 0\}$ is finite; equivalently, every
critical point of $f$ is preperiodic. (See, for example, \cite[Theorem~3.1]{Jones}.)

For any given PCF map, one may ask whether there is 
an associated subgroup of $\Aut(T_{d,\infty})$
that always contains, and in some cases equals, the arboreal Galois group $G_{\infty}$.
In \cite{BFHJY}, this question was answered in the affirmative for the PCF cubic polynomial
$-2z^3+3z^2$, including an explicit computation of the subgroup
$E_{\infty}\subsetneq \Aut(T_{3,\infty})$
and a simple sufficient condition on $K$ and $x_0$ for $G_{\infty}$ to be all of $E_{\infty}$.
In the present paper, we do the same for the PCF quadratic polynomial $f(z):=z^2-1$.

Let $T_{\infty}$ and $T_{n}$ denote
the binary rooted trees $T_{2,\infty}$ and $T_{2,n}$, respectively.
The two critical points $0,\infty$ of $f$ are both periodic, with $\infty\mapsto\infty$
and $0\mapsto -1 \mapsto 0$. Over the function field
$K=\CC(t)$ with $x_0=t$,
a setting in which arboreal Galois groups are often known as
 \emph{profinite iterated monodromy groups},
$G_{\infty}$ is isomorphic to the closure $\overline{B}_{\infty}$ of a well-understood subgroup
$B_{\infty}$ of $\Aut(T_\infty)$ called the \emph{basilica group}.
(See \cite[Section~6.12.1]{Nek}, as well as
\cite[Section~5]{BGN}, especially Theorem~5.8 and following.)
Here and throughout this paper, when we say that two groups that act on a tree
are isomorphic, we mean not only that they are isomorphic as abstract groups,
but also that the isomorphism is equivariant with respect to the action on the tree.

In \cite[Theorem~2.5.6]{Pink}, Pink showed for \emph{any} algebraically
closed field $\kbar$ not of characteristic~2, then with $K=\kbar(t)$ and $x_0=t$,
the arboreal Galois group $G_{\infty}$ is isomorphic to $\overline{B}_{\infty}$.
Pink also showed that for function fields $K=k(t)$ where $k$ is \emph{not}
algebraically closed, the arboreal Galois group $G_{\infty}$ is an extension of
$\overline{B}_{\infty}$ by a subgroup of the $2$-adic multiplicative group
$\ZZ_2^{\times}$, via the $2$-adic cyclotomic character $\Gal(\kbar/k)\to\ZZ_2^{\times}$,
which factors through $G_{\infty}$.
(See \cite[Theorem~2.8.4]{Pink}.)
We define and discuss $\overline{B}_{\infty}$ in Section~\ref{sec:Basil}.

However,
our interest in this paper extends to the case that the field $K$ is a number field, 
which is not directly covered by Pink's work \cite{Pink}.
More precisely, by choosing the coefficient field $k$ in \cite{Pink} to be a number field,
and specializing the parameter $t$ in $K=k(t)$ to some $x_0\in k$, the resulting
arboreal Galois group $\Gal(k_{\infty}/k)$ is a decomposition subgroup
of the generic Galois group $\Gal(K_{\infty}/K)$.
Since our aims are to determine when that subgroup is the full group,
and to present a new description of that group,
we pursue a different strategy, as follows.

We give an explicit definition of a subgroup $M_{\infty}\subseteq \Aut(T_{\infty})$
that we call the \emph{arithmetic basilica group}
and which is an extension of $\overline{B}_{\infty}$ by $\ZZ_2^{\times}$.
We do so by defining a quantity $P(\sigma,x)\in\ZZ_2^{\times}$
for each $\sigma\in\Aut(T_{\infty})$ and each node $x$ of the tree $T_{\infty}$;
see Section~\ref{sec:Pdef}.
Then, in Section~\ref{sec:Mn}, we define $M_{\infty}$ to be the set of such $\sigma$ for which
$P(\sigma,x)$ has that same value for each node $x$.
In the remaining sections, we prove that $M_{\infty}$ is the desired group.
Here is a formal statement summarizing our main results; as noted above,
the group isomorphisms here are assumed to be equivariant with respect to the action
on the associated trees. More precisely, when we say that $G_{\infty}$
is isomorphic to a given subgroup of $\Aut(T_\infty)$, we mean $T_{\infty}$
can be mapped onto the tree of preimages of $x_0$ so as to yield this equivariance.

\begin{mthm}
Let $K$ be a field of characteristic different from $2$,
and let $x_0\in K$ with $x_0\neq 0,-1$.
Let $G_{\infty}$ be the arboreal Galois group for $f(z)=z^2-1$ over $K$, rooted at $x_0$.
Then:
\begin{enumerate}
\item $G_{\infty}$ is isomorphic to a subgroup of
the arithmetic basilica group $M_{\infty}$.
\item The following are equivalent:
\begin{enumerate}
\item $G_{\infty}\cong M_{\infty}$.
\item $G_5 \cong M_5$.
\item $[K(\sqrt{-x_0},\sqrt{1+x_0},\zeta_{8}):K]=16$.
\end{enumerate}
\end{enumerate}
\end{mthm}

Here,
$M_n$ denotes the quotient of $M_{\infty}$
formed by restricting to its action on the subtree $T_n$,
and 
$\zeta_8$ denotes a primitive eighth root of unity.

The above theorem shows that, like the map $z\mapsto -2z^3+3z^2$ of \cite{BFHJY},
the PCF map $f(z)=z^2-1$ has an
associated subgroup $M_{\infty}\subsetneq \Aut(T_{\infty})$
that always contains and sometimes equals the arboreal Galois group $G_{\infty}$.
Condition (2b) shows that this equality is attained for the entire tree if it is
already attained at the fifth level, and condition~(2c) is very easy to check in practice.

We note that if $[K(\zeta_8):K]=4$, then by Hilbert's irreducibility theorem,
there are many choices of $x_0\in K$ for which $[K(\sqrt{-x_0},\sqrt{1+x_0},\zeta_{8}):K]=16$,
since the set of $x_0\in K$ failing this condition is a thin set, in the sense of Serre.
For example, if $K=\QQ$, then the condition holds for
\[ x_0 \text{ or } -1-x_0 \text{ in } \{5,6,10,11,12,13,14,19,20,21,22,23\ldots \} \]
among (infinitely) many other examples.

On the other hand, even when $[K(\sqrt{-x_0},\sqrt{1+x_0},\zeta_{8}):K]<16$,
our computations suggest the following conjecture.
\begin{conj}
\label{conj:basic}
Let $K$ be a number field.
Then for all but finitely many choices of $x_0\in K$,
the associated arboreal Galois group
$G_{\infty}$ for $f(z)=z^2-1$ has finite index in $M_{\infty}$.
\end{conj}
We must allow for finitely many exceptional $x_0$ in Conjecture~\ref{conj:basic};
for example, it is not hard to see that $[M_{\infty}:G_{\infty}]=\infty$
if $x_0$ is periodic.
More generally, in light of our main theorem and the results of \cite{BFHJY},
as well as \cite[Conjecture~1.1]{JonMan} and \cite[Theorem~1.1]{BHL},
we propose the following broader conjecture.
\begin{conj}
\label{conj:general}
Let $k$ be a number field, and
let $\phi(z)\in k(z)$ be a rational function of degree $d\geq 2$,
defined over $k$.
Let $L$ be the function field $L:=k(t)$,
and let $G(\phi,k)$ be the arboreal Galois group of $\phi$ over $L$
with root point $t$; that is,
\[ G(\phi,k) :=\Gal(L_{\infty}/L), \quad \text{where} \quad
L_{\infty} :=\bigcup_{n\geq 0} L\big(\phi^{-n}(t)\big) .\]
Then for any finite extension $K/k$ and any $x_0\in\PP^1(K)$,
the associated arboreal Galois group $G_{\infty}:=\Gal(K_{\infty}/K)$
is isomorphic to a subgroup of $G(\phi,k)$.
Moreover, with $K=k$, it is possible to choose $x_0$
so that $G_{\infty}$ is the full group $G(\phi,k)$.
\end{conj}
The first conclusion of Conjecture~\ref{conj:general} is an immediate consequence
of the fact that when we specialize to $t=x_0$, the resulting Galois group $G_{\infty}$
is the associated decomposition subgroup of $G(\phi,k)$.
Thus, the main content of Conjecture~\ref{conj:general} is the second conclusion: that
we can choose $x_0$ so that the inclusion of $G_{\infty}$ in $G(\phi,k)$
is an (equivariant) isomorphism.

%
In the notation of Conjecture~\ref{conj:general}, one can also ask
for sufficient conditions that $G_{\infty}$ has finite index in $G(\phi,k)$.
Besides periodic $x_0$, we also have $[G(\phi,k):G_{\infty}]=\infty$
if some $\Orb_\phi^-(x_0)$ contains a critical point of $\phi$;
if $\phi$ is not PCF, then this can happen for an infinite (but thin) set of $x_0\in K$.
Another example arises for $\phi(z)=z^2$: for $x_0=-1$,
we have $K_{\infty}=L$, where $L=K(\zeta_{2^{\infty}})$,
but for $x_0=3$, we have $K_{\infty}=L(3^{1/2^{\infty}})$,
which is an infinite extension of $L$.

The outline of the paper is as follows.
In Section~\ref{sec:zeta}, we we discuss labelings of a binary tree
and explicitly construct 2-power roots of unity in $K_{\infty}$.
Motivated by this construction, in Section~\ref{sec:Pdef} we define the quantity
$P(\sigma,x)\in\ZZ_2$ that we mentioned just before our Main Theorem above,
and in Section~\ref{sec:Mn},
we use $P$ to define our arithmetic basilica group $M_{\infty}$.
In Section~\ref{sec:Basil}, we recall the definition and properties
of the closed basilica group $\overline{B}_{\infty}$,
and we study the finite groups $M_n$ and $B_n$ formed by restricting
$M_{\infty}$ and $\overline{B}_{\infty}$ to the finite subtree $T_n$.
We also prove Theorems~\ref{thm:Bndef} and~\ref{thm:MBmap},
that $P:M_{\infty}\to\ZZ_2^{\times}$ is surjective, with kernel $\overline{B}_{\infty}$.
Finally, Section~\ref{sec:mainproof} is devoted to the proof
of statement~(2) of our Main Theorem and related results.

\section{Realizing roots of unity}
\label{sec:zeta}
In this section, we explicitly construct $2$-power roots of unity
from iterated preimages of $f(z)=z^2-1$.
First, we set the following notation.
\begin{tabbing}
\hspace{8mm} \= \hspace{15mm} \=  \kill
\> $K$: \> a field of characteristic different from $2$, with algebraic closure $\Kbar$ \\
\> $f$: \> the polynomial $f(z)=z^2-1$ \\
\> $x_0$: \> an element of $K$, to serve as the root of our preimage tree \\
\> $T_n$: \> a binary rooted tree, extending $n$ levels above its root node \\
\> $T_\infty$: \> a binary rooted tree, extending infinitely above its root node \\
\> $K_n$: \> for each $n\geq 0$, the extension field $K_n:=K(f^{-n}(x_0))$ \\
\> $K_\infty$: \> the union $K_{\infty} = \bigcup_{n\geq 1} K_n$ in $\Kbar$ \\
\> $G_n$: \> the Galois group $\Gal(K_n/K_0)$ \\
\> $G_{\infty}$: \> the Galois group $\Gal(K_\infty/K_0)$
\end{tabbing}
We will often abuse notation and write $x_0$ for both a chosen point in $K=K_0$,
and for the root point of the abstract tree $T_n$ or $T_{\infty}$. The $2^m$ nodes at the
$m$-th level of the abstract tree then correspond to the points of $f^{-m}(x_0)$
(repeated with appropriate multiplicity if $x_0$ is in the critical orbit), with a particular
node $w$ lying directly above another node $y$ if and only if $f(w)=y$.

\begin{figure}
\begin{tikzpicture}
\path[draw] (.2,2) -- (.8,1.2) -- (1.4,2);
\path[draw] (2.6,2) -- (3.2,1.2) -- (3.8,2);
\path[draw] (.8,1.2) -- (2,.4) -- (3.2,1.2);
\path[fill] (.2,2) circle (0.1);
\path[fill] (1.4,2) circle (0.1);
\path[fill] (2.6,2) circle (0.1);
\path[fill] (3.8,2) circle (0.1);
\path[fill] (.8,1.2) circle (0.1);
\path[fill] (3.2,1.2) circle (0.1);
\path[fill] (2,.4) circle (0.1);
\node (y) at (1.7,.1) {$y$};
\node (u) at (.4,1) {$u$};
\node (negu) at (3.6,1) {$-u$};
\node (a1) at (-.2,1.9) {$\alpha_1$};
\node (a2) at (2.2,1.9) {$\alpha_2$};
\end{tikzpicture}
\caption{Lemma~\ref{lem:2down}: $(\alpha_1 \alpha_2)^2 = -y$}
\label{fig:2down}
\end{figure}
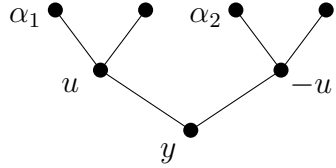

\begin{lemma}
\label{lem:2down}
Let $y\in\Kbar$, and let $\alpha_1,\alpha_2\in f^{-2}(y)$ with
$f(\alpha_2)=-f(\alpha_1)$. Then
\[ (\alpha_1 \alpha_2)^2 = -y . \]
\end{lemma}

\begin{proof}
Write $u=f(\alpha_1)$. Then $\alpha_1^2=u+1$ and $\alpha_2^2 = -u+1$.
See Figure~\ref{fig:2down}, showing the tree $T_2$ rising above $y$,
which has preimages $u$ and $-u$, which in turn have preimages $\pm \alpha_1$
and $\pm \alpha_2$, respectively.
Thus,
\[ (\alpha_1 \alpha_2)^2 = (u+1)(-u+1) = -(u^2-1) = -f(u) = -y. \qedhere \]
\end{proof}

\begin{lemma}
\label{lem:nrel}
Let $m\geq 0$, and let $y\in\Kbar$ with $y\neq 0,-1$.
Let $\alpha_{0,1}\in f^{-1}(y)$, and define $\beta_{0,1}:=-\alpha_{0,1}$.
For each $i=1,\ldots, m$, choose points 
$\{\alpha_{i,j} , \beta_{i,j}: 1\leq j \leq 2^{i}\}\subseteq f^{-(2i+1)}(y)$ so that
\[ f^{-2}\big(\alpha_{i-1,\ell}\big) = \{ \pm \alpha_{i,{2\ell -1}}, \pm \alpha_{i,{2\ell}} \}
\quad\text{and}\quad
f^{-2}\big(\beta_{i-1,\ell}\big) = \{ \pm \beta_{i,{2\ell -1}}, \pm \beta_{i,{2\ell}} \} \]
for each $\ell=1,\ldots,2^{i-1}$,
as in Figure~\ref{fig:ablevel5}. Define
\[ \gamma_m:= \prod_{j=1}^{2^{m}} \alpha_{m,j}
\quad\text{and}\quad
\delta_m:= \prod_{j=1}^{2^{m}} \beta_{m,j} . \]
Then $\dsps (-\gamma_m)^{2^{m}} = \beta_{0,1}$
and $\dsps (-\delta_m)^{2^{m}} = \alpha_{0,1}$.
Moreover,
$\gamma_m/\delta_m$ is a primitive $2^{m+1}$-th root of unity.
\end{lemma}

\begin{proof} The conclusion is trivially true for $m=0$.
Proceeding inductively, consider $m\geq 1$,
and assume it holds for $m-1$.
For each $\ell=1,\ldots, 2^{m-1}$, we have
\[ \big(\alpha_{m,2\ell-1} \alpha_{m,2\ell} \big)^2 = -\alpha_{m-1,\ell}
\quad\text{and}\quad \big(\beta_{m,2\ell-1} \beta_{m,2\ell} \big)^2 = -\beta_{m-1,\ell} \]
by Lemma~\ref{lem:2down}. It follows immediately that
\[ (-\gamma_m)^2 = 
\begin{cases}
\gamma_{m-1} & \text{ if } m\geq 2, \\
-\gamma_{m-1} & \text{ if } m=1,
\end{cases}
\quad\text{and}\quad
(-\delta_m)^2 = \begin{cases}
\delta_{m-1} & \text{ if } m\geq 2, \\
-\delta_{m-1} & \text{ if } m=1.
\end{cases} \]
Raising each to the power $2^{m-1}$, which is $1$ if $m=1$ and even for $m\geq 2$, we have
\[ (-\gamma_m)^{2^{m}} = (-\gamma_{m-1})^{2^{m-1}} = \beta_{0,1}
\quad\text{and}\quad
(-\delta_m)^{2^{m}} = (-\delta_{m-1})^{2^{m-1}} = \alpha_{0,1}, \]
as desired. The final statement is immediate from the fact that $\alpha_{0,1}\neq 0$
(since $y\neq -1$), and hence $\beta_{0,1}/\alpha_{0,1}=-1$.
\end{proof}

\begin{figure}
\begin{tikzpicture}
\path[draw] (0.2,4) -- (0.4,3.2) -- (0.6,4);
\path[draw] (1,4) -- (1.2,3.2) -- (1.4,4);
\path[draw] (1.8,4) -- (2,3.2) -- (2.2,4);
\path[draw] (2.6,4) -- (2.8,3.2) -- (3,4);
\path[draw] (3.4,4) -- (3.6,3.2) -- (3.8,4);
\path[draw] (4.2,4) -- (4.4,3.2) -- (4.6,4);
\path[draw] (5,4) -- (5.2,3.2) -- (5.4,4);
\path[draw] (5.8,4) -- (6,3.2) -- (6.2,4);
\path[draw] (6.6,4) -- (6.8,3.2) -- (7,4);
\path[draw] (7.4,4) -- (7.6,3.2) -- (7.8,4);
\path[draw] (8.2,4) -- (8.4,3.2) -- (8.6,4);
\path[draw] (9,4) -- (9.2,3.2) -- (9.4,4);
\path[draw] (9.8,4) -- (10,3.2) -- (10.2,4);
\path[draw] (10.6,4) -- (10.8,3.2) -- (11,4);
\path[draw] (11.4,4) -- (11.6,3.2) -- (11.8,4);
\path[draw] (12.2,4) -- (12.4,3.2) -- (12.6,4);
\path[fill,gray] (0.2,4) circle (0.19);
\path[fill] (0.6,4) circle (0.1);
\path[fill,gray] (1,4) circle (0.19);
\path[fill] (1.4,4) circle (0.1);
\path[fill] (1.8,4) circle (0.1);
\path[fill] (2.2,4) circle (0.1);
\path[fill] (2.6,4) circle (0.1);
\path[fill] (3,4) circle (0.1);
\path[fill,gray] (3.4,4) circle (0.19);
\path[fill] (3.8,4) circle (0.1);
\path[fill,gray] (4.2,4) circle (0.19);
\path[fill] (4.6,4) circle (0.1);
\path[fill] (5,4) circle (0.1);
\path[fill] (5.4,4) circle (0.1);
\path[fill] (5.8,4) circle (0.1);
\path[fill] (6.2,4) circle (0.1);
\path[fill,gray] (6.6,4) circle (0.19);
\path[fill] (7,4) circle (0.1);
\path[fill,gray] (7.4,4) circle (0.19);
\path[fill] (7.8,4) circle (0.1);
\path[fill] (8.2,4) circle (0.1);
\path[fill] (8.6,4) circle (0.1);
\path[fill] (9,4) circle (0.1);
\path[fill] (9.4,4) circle (0.1);
\path[fill,gray] (9.8,4) circle (0.19);
\path[fill] (10.2,4) circle (0.1);
\path[fill,gray] (10.6,4) circle (0.19);
\path[fill] (11,4) circle (0.1);
\path[fill] (11.4,4) circle (0.1);
\path[fill] (11.8,4) circle (0.1);
\path[fill] (12.2,4) circle (0.1);
\path[fill] (12.6,4) circle (0.1);
\path[draw] (0.4,3.2) -- (0.8,2.4) -- (1.2,3.2);
\path[draw] (2,3.2) -- (2.4,2.4) -- (2.8,3.2);
\path[draw] (3.6,3.2) -- (4,2.4) -- (4.4,3.2);
\path[draw] (5.2,3.2) -- (5.6,2.4) -- (6,3.2);
\path[draw] (6.8,3.2) -- (7.2,2.4) -- (7.6,3.2);
\path[draw] (8.4,3.2) -- (8.8,2.4) -- (9.2,3.2);
\path[draw] (10,3.2) -- (10.4,2.4) -- (10.8,3.2);
\path[draw] (11.6,3.2) -- (12,2.4) -- (12.4,3.2);
\path[fill] (0.4,3.2) circle (0.1);
\path[fill] (1.2,3.2) circle (0.1);
\path[fill] (2.0,3.2) circle (0.1);
\path[fill] (2.8,3.2) circle (0.1);
\path[fill] (3.6,3.2) circle (0.1);
\path[fill] (4.4,3.2) circle (0.1);
\path[fill] (5.2,3.2) circle (0.1);
\path[fill] (6,3.2) circle (0.1);
\path[fill] (6.8,3.2) circle (0.1);
\path[fill] (7.6,3.2) circle (0.1);
\path[fill] (8.4,3.2) circle (0.1);
\path[fill] (9.2,3.2) circle (0.1);
\path[fill] (10,3.2) circle (0.1);
\path[fill] (10.8,3.2) circle (0.1);
\path[fill] (11.6,3.2) circle (0.1);
\path[fill] (12.4,3.2) circle (0.1);
\path[draw] (0.8,2.4) -- (1.6,1.6) -- (2.4,2.4);
\path[draw] (4,2.4) -- (4.8,1.6) -- (5.6,2.4);
\path[draw] (7.2,2.4) -- (8,1.6) -- (8.8,2.4);
\path[draw] (10.4,2.4) -- (11.2,1.6) -- (12,2.4);
\path[fill,gray] (0.8,2.4) circle (0.19);
\path[fill] (2.4,2.4) circle (0.1);
\path[fill,gray] (4,2.4) circle (0.19);
\path[fill] (5.6,2.4) circle (0.1);
\path[fill,gray] (7.2,2.4) circle (0.19);
\path[fill] (8.8,2.4) circle (0.1);
\path[fill,gray] (10.4,2.4) circle (0.19);
\path[fill] (12,2.4) circle (0.1);
\path[draw] (1.6,1.6) -- (3.2,0.8) -- (4.8,1.6);
\path[draw] (8,1.6) -- (9.6,0.8) -- (11.2,1.6);
\path[fill] (1.6,1.6) circle (0.1);
\path[fill] (4.8,1.6) circle (0.1);
\path[fill] (8,1.6) circle (0.1);
\path[fill] (11.2,1.6) circle (0.1);
\path[draw] (3.2,0.8) -- (6.4,0.2) -- (9.6,0.8);
\path[fill,gray] (3.2,0.8) circle (0.19);
\path[fill,gray] (9.6,0.8) circle (0.19);
\path[fill] (6.4,0.2) circle (0.1);
\node (y) at (6.75,0) {$y$};
\node (a0) at (2.7,0.6) {$\alpha_{0,1}$};
\node (b0) at (10.9,0.6) {$\beta_{0,1}=-\alpha_{0,1}$};
\node (a21) at (.5,2) {$\alpha_{1,1}$};
\node (a22) at (3.7,2) {$\alpha_{1,2}$};
\node (b21) at (7,2) {$\beta_{1,1}$};
\node (b22) at (10.2,2) {$\beta_{1,2}$};
\node (a31) at (0,4.4) {$\alpha_{2,1}$};
\node (a32) at (1,4.4) {$\alpha_{2,2}$};
\node (a33) at (3.3,4.4) {$\alpha_{2,3}$};
\node (a34) at (4.3,4.4) {$\alpha_{2,4}$};
\node (b31) at (6.5,4.4) {$\beta_{2,1}$};
\node (b32) at (7.4,4.4) {$\beta_{2,2}$};
\node (b33) at (9.7,4.4) {$\beta_{2,3}$};
\node (b34) at (10.7,4.4) {$\beta_{2,4}$};
\end{tikzpicture}
\caption{Lemma~\ref{lem:nrel} for $m=2$.}
\label{fig:ablevel5}
\end{figure}
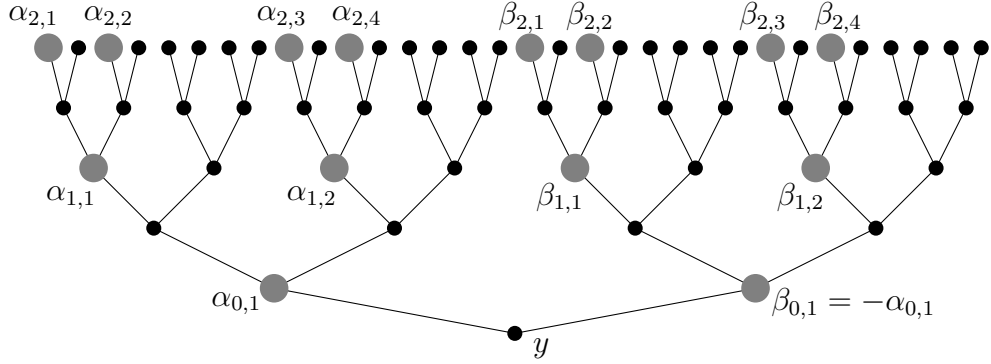

In light of the careful selection of points in Lemma~\ref{lem:nrel}, i.e.,
the highlighted nodes in Figure~\ref{fig:ablevel5}, we will need a precise
system for labeling the nodes of a binary tree, as follows.

\begin{defin}
\label{def:labeling}
A \emph{labeling} of $T_{\infty}$ is a choice of
two tree morphisms $a,b:T_{\infty}\to T_{\infty}$ such that
$a$ maps $T_{\infty}$ bijectively onto the subtree rooted at one
of the two nodes connected to the root node $x_0$,
and $b$ maps $T_{\infty}$ bijectively onto the subtree rooted at the other.

For any integer $n\geq 1$, a labeling of $T_n$ is a choice of
two injective tree morphisms $a,b:T_{n-1}\to T_n$
with the same property.
\end{defin}

To see why the choice of maps $a,b$ in Definition~\ref{def:labeling}
should be considered a ``labeling'' of each node of the tree,
consider a node $y$ at the $m$-th level of $T_{\infty}$.
By our choice of the maps $a,b$,
there is a unique ordered $m$-tuple $(s_1,\ldots,s_m) \in \{a,b\}^m$
such that $y=s_1\circ \cdots \circ s_m (x_0)$.
Thus, it makes sense to label the node $y$ with the $m$-tuple
$(s_1,\ldots,s_m)$. The node directly underneath $y$
then has label $(s_1,\ldots,s_{m-1})$.
We will usually dispense with the punctuation and write
$s_1 s_2 \cdots s_m$ instead of $(s_1,\ldots,s_m)$.
We will also frequently abuse notation and
refer to a node $y$ and its label in $\{a,b\}^m$ interchangeably.

Note that the order we have written the $m$-tuple $(s_1,\ldots,s_m)$
is also the order we trace
up the tree when following the path from $x_0$ to $y$. That is, $s_1$
tells us whether to go left ($a$) or right ($b$) to get from the root node
to level~$1$; $s_2$ tells us whether to go left or right from there to level~$2$;
and so on until we arrive at $y$.
See Figure~\ref{fig:treelabel}.

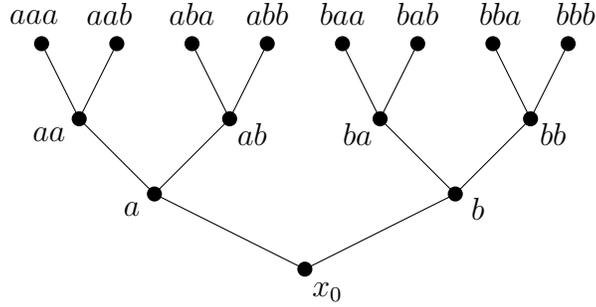
\begin{figure}
\begin{tikzpicture}
\path[draw] (.4,3.5) -- (.9,2.5) -- (1.4,3.5);
\path[draw] (2.4,3.5) -- (2.9,2.5) -- (3.4,3.5);
\path[draw] (4.4,3.5) -- (4.9,2.5) -- (5.4,3.5);
\path[draw] (6.4,3.5) -- (6.9,2.5) -- (7.4,3.5);
\path[draw] (.9,2.5) -- (1.9,1.5) -- (2.9,2.5);
\path[draw] (4.9,2.5) -- (5.9,1.5) -- (6.9,2.5);
\path[draw] (1.9,1.5) -- (3.9,0.5) -- (5.9,1.5);
\path[fill] (.4,3.5) circle (0.1);
\path[fill] (1.4,3.5) circle (0.1);
\path[fill] (2.4,3.5) circle (0.1);
\path[fill] (3.4,3.5) circle (0.1);
\path[fill] (4.4,3.5) circle (0.1);
\path[fill] (5.4,3.5) circle (0.1);
\path[fill] (6.4,3.5) circle (0.1);
\path[fill] (7.4,3.5) circle (0.1);
\path[fill] (0.9,2.5) circle (0.1);
\path[fill] (2.9,2.5) circle (0.1);
\path[fill] (4.9,2.5) circle (0.1);
\path[fill] (6.9,2.5) circle (0.1);
\path[fill] (1.9,1.5) circle (0.1);
\path[fill] (5.9,1.5) circle (0.1);
\path[fill] (3.9,0.5) circle (0.1);
\node (a) at (1.6,1.3) {$a$};
\node (b) at (6.2,1.3) {$b$};
\node (aa) at (0.5,2.3) {$aa$};
\node (ab) at (3.2,2.3) {$ab$};
\node (ba) at (4.6,2.3) {$ba$};
\node (bb) at (7.2,2.3) {$bb$};
\node (aaa) at (0.3,3.85) {$aaa$};
\node (aab) at (1.3,3.9) {$aab$};
\node (aba) at (2.4,3.9) {$aba$};
\node (abb) at (3.4,3.9) {$abb$};
\node (baa) at (4.4,3.9) {$baa$};
\node (bab) at (5.4,3.9) {$bab$};
\node (bba) at (6.5,3.9) {$bba$};
\node (bbb) at (7.5,3.9) {$bbb$};
\node (x0) at (4.2,0.2) {$x_0$};
\end{tikzpicture}
\caption{A labeling of $T_3$}
\label{fig:treelabel}
\end{figure}

By Lemma~\ref{lem:nrel}, the field $K_{\infty}$ formed by
adjoining all preimages $f^{-n}(x_0)$ to $K_0$ contains all $2$-power roots of unity.
Our next result shows that we can choose a labeling to make this statement more precise.

\begin{lemma}
\label{lem:pickzeta}
Let $x_0\in K$ with $x_0\neq 0,-1$.
Choose a sequence $\{ \zeta_2, \zeta_4,\zeta_8,\ldots\}$
of primitive $2$-power roots of unity in $\Kbar$,
in such a way that $(\zeta_{2^{m}})^2 = \zeta_{2^{m-1}}$
for each $m\geq 1$. It is possible to label the tree $T_{\infty}$
of preimages $\Orb_f^-(x_0)$ in such a way that
for every node $y$ of the tree and for every $i\geq 0$, we have
\begin{equation}
\label{eq:zetaprod}
\Bigg( \prod_{s_1,\ldots, s_{i}\in\{a,b\}} [yas_1 as_2 \cdots as_{i}a] \Bigg)
\Bigg( \prod_{s_1,\ldots, s_{i}\in\{a,b\}} [ybs_1 as_2 \cdots as_{i}a] \Bigg)^{-1} = \zeta_{2^{i+1}} ,
\end{equation}
where $[w]$ denotes the element of $\Kbar$ that appears in the tree as a node with label $w$.
\end{lemma}

Lemma~\ref{lem:pickzeta} says that it is always possible to choose a labeling
of the tree of preimages $\Orb_f^-(x_0)$ so that
for the nodes $\alpha_{i,j}$ and $\beta_{i,j}$ highlighted
as in Figure~\ref{fig:ablevel5}, we have
\[ \frac{\alpha_{1,1} \alpha_{1,2}}{\beta_{1,1}\beta_{1,2}}=\zeta_4, \quad
\frac{\alpha_{2,1} \alpha_{2,2}\alpha_{2,3} \alpha_{2,4}}
{\beta_{2,1}\beta_{2,2} \beta_{2,3}\beta_{2,4}}=\zeta_8, \quad
\text{ and so on}.\]
In fact, it says that we can label the tree so that these relationships hold simultaneously
for every subtree of the full tree.
By contrast, Lemma~\ref{lem:nrel} says that after applying an arbitrary automorphism $\tau$
of $T_{\infty}$, which is roughly equivalent to choosing an arbitrary labeling of the tree,
any such product of elements of $f^{-(2m+1)}(y)$ is \emph{some} primitive
$2^{m+1}$-root of unity, albeit not necessarily the particular root $\zeta_{2^{m+1}}$.
Finding precisely which power of $\zeta_{2^{m+1}}$ would then be given
by the resulting variant of expression~\eqref{eq:zetaprod} would be a finite
but presumably stringy computation. Thus, our strategy in later sections
will be to use Lemma~\ref{lem:pickzeta}
to fix a labeling of the tree once and for all.

\begin{proof}[Proof of Lemma~\ref{lem:pickzeta}]
The assumption that $x_0\neq 0,-1$ implies that no node $y\in\Orb_f^-(x_0)$ is $0$ or $-1$.
We will label the tree of preimages inductively, starting from the root point $x_0$
and working our way up.
To begin, label the two preimages of $x_0$ arbitrarily as $a$ and $b$.
Similarly, arbitrarily label the two preimages of $a$ as $aa$ and $ab$,
and the two preimages of $b$ as $ba$ and $bb$. Thus, we have a labeling on the copy
of $T_2$ rooted at $x_0$. For each of the nodes $y\in\{x_0,a,b\}$,
we have $(ya)/(yb)=-1=\zeta_2$.
Thus, the desired identity~\eqref{eq:zetaprod} holds at every node of $T_1$ for $i=0$.

For each successive $n\geq 3$,
suppose that we have labeled $T_{n-1}$ in such a way that for every node $y$
at every level $0\leq \ell\leq n-2$ of $T_{n-1}$,
and for every $0\leq i \leq \lfloor (n-\ell-2)/2 \rfloor$,
the identity of equation~\eqref{eq:zetaprod} holds.
For each node $x$ at level $n-1$, label the two points
of $f^{-1}(x)$ arbitrarily as $xa$ and $xb$. We will now adjust these labels that we have just
applied at the $n$-th level of the tree.

If $n=2m+1$ is odd,
consider the product on the left side of equation~\eqref{eq:zetaprod} for $y=x_0$, with $i=m$;
or if $n=2m+2$ is even,
consider this product for each of $y=a$ and $y=b$, with $i=m$.
As in the proof of Lemma~\ref{lem:nrel}, it is immediate from Lemma~\ref{lem:2down}
that the square of this product is precisely the corresponding
quantity for $y$ with $i=m-1$. (When $m=1$, each half has a negative sign,
but the negatives cancel in that case.)
By our successful labeling of $T_{n-1}$, this square is $\zeta_{2^{m}}$.
Thus, the original product is $\pm\zeta_{2^{m+1}}$. If it is $-\zeta_{2^{m+1}}$, exchange the labels
of the two level-$n$ nodes $ybaa\cdots aa$ and $ybaa\cdots ab$; otherwise, make no
label changes for now. Since these two points in $f^{-n}(x_0)$ are negatives of each other,
equation~\eqref{eq:zetaprod} now holds for $y$ with $i=m$.

Next, consider the product on the left side of equation~\eqref{eq:zetaprod} with $i=m-1$
for each node $y$ at level $2$ of the tree (if $n=2m+1$ is odd) or at level $3$ (if $n=2m+2$ is even).
By Lemma~\ref{lem:2down} and our labeling of $T_{n-1}$ again,
the square of this product is $\zeta_{2^{m-1}}$, and hence the product itself
is $\pm\zeta_{2^{m}}$. If it is $-\zeta_{2^{m}}$, exchange the labels
of the two level-$n$ nodes $ybaa\cdots aa$ and $ybaa\cdots ab$; otherwise, make no
label changes for now. Since these two points in $f^{-n}(x_0)$ are negatives of each other,
equation~\eqref{eq:zetaprod} now holds for $y$ with $i=m-1$.
In addition, because both of these nodes have labels beginning $yb\cdots$, they did
not show up in the product of equation~\eqref{eq:zetaprod} for nodes strictly lower on the tree
than $y$, so we have not disrupted our previous arrangements.

Continue in this fashion, considering nodes at successive even levels $\ell$ of the tree
(if $n$ is odd) or odd levels $\ell$ of the tree (if $n$ is even). For each such node $y$,
choose whether or not to switch the labels of $ybaa\cdots aa$ and $ybaa\cdots ab$
to ensure that equation~\eqref{eq:zetaprod} holds for $y$ with $i=(n-\ell-1)/2$.
Once we have finished working our way up
through level $\ell=n-1$, we have labeled $T_n$
so that for every node $y$ at every level $0\leq \ell\leq n-1$ of $T_n$,
and for every $0\leq i \leq \lfloor (n-\ell-1)/2 \rfloor$,
the identity of equation~\eqref{eq:zetaprod} holds.
Thus, our inductive construction is complete.
\end{proof}

\section{A special infinite sum on the tree}
\label{sec:Pdef}
Any tree automorphism
$\sigma\in\Aut(T_\infty)$ or $\sigma\in\Aut(T_n)$
must permute the $2^m$ nodes at each level $m$ of the tree.
Moreover, since $\sigma$ must preserve the tree structure,
for each node $x$ of the (labeled) tree, it must map the set
$\{xa,xb\}$ of two nodes above $x$ to the set $\{\sigma(x)a,\sigma(x)b\}$
of two nodes above $\sigma(x)$.
Thus, for any tree automorphism $\sigma$ and $m$-tuple $x\in\{a,b\}^m$, we define 
the \emph{parity} $\Par(\sigma,x)\in\ZZ$ of $\sigma$ at $x$ to be
\begin{equation}
\label{eq:sgndef}
\Par(\sigma,x) := \begin{cases}
0 & \text{ if } \sigma(xa)=\sigma(x)a \text{ and } \sigma(xb)=\sigma(x)b
\\
1 & \text{ if } \sigma(xa)=\sigma(x)b \text{ and } \sigma(xb)=\sigma(x)a
\end{cases}
\end{equation}
Observe that any set of choices of $\Par(\sigma,x)$ for each node
$x$ of $T_{\infty}$ (respectively, $T_{n-1}$) determines a
unique automorphism $\sigma\in \Aut(T_{\infty})$ (respectively, $\sigma\in\Aut(T_n)$).

If $\sigma(x)=x$, then $\Par(\sigma,x)$ is $0$ if $\sigma$ fixes
the two nodes above $x$, or $1$ if it transposes them. However, $\Par(\sigma,x)$
is defined even when $\sigma(x)\neq x$, although in that case its value depends also on the
labeling of the tree.

For a Galois element $\sigma\in G_n$ or $\sigma\in G_{\infty}$, considered as
acting on the tree $T_n$ or $T_{\infty}$ of preimages of $x_0$, we will need to understand
the action of $\sigma$ on $2$-power roots of unity $\zeta$. Specifically,
we must have $\sigma(\zeta)=\zeta^P$ for some $P=P(\sigma)$
in the group $(\ZZ/2^j\ZZ)^\times$ (if $\sigma\in G_n$, where $j:=\lfloor (n+1)/2 \rfloor$),
or in $\ZZ_2^{\times}$ (if $\sigma\in G_{\infty}$).
Lemmas~\ref{lem:nrel} and~\ref{lem:pickzeta} inspire the following candidate for this power $P$.

\begin{defin}
\label{def:Pmap}
Fix a labeling of $T_{\infty}$, and let $\sigma\in\Aut(T_{\infty})$.
For any node $x$ of $T_{\infty}$, define
\begin{equation}
\label{eq:Qdef}
Q(\sigma,x):= \sum_{i\geq 0} 2^i \sum_{s_1,\ldots,s_i\in\{a,b\}}
\Par(\sigma,xas_1as_2\cdots as_i) \in \ZZ_2,
\end{equation}
and
\begin{equation}
\label{eq:Pdef}
P(\sigma,x):= (-1)^{\Par(\sigma,x)}  + 2\sum_{t\in\{a,b\}} Q(\sigma,xbt) - 
2\sum_{t\in\{a,b\}}Q(\sigma,xat) \in \ZZ_2^{\times} .
\end{equation}
In addition, for any $n\geq m\geq 0$, any node $x$ at level $m$ of $T_n$,
and any $\tau\in\Aut(T_n)$, set $j:=\lfloor (n-m+1)/2 \rfloor$,
and define $P(\tau,x)\in(\ZZ/2^j\ZZ)^{\times}$ to be
\begin{equation}
\label{eq:Pdefmod}
P(\tau,x):\equiv P(\tilde{\tau},x) \pmod{2^j},
\end{equation}
where $\tilde{\tau}\in\Aut(T_{\infty})$ is any extension of $\tau$ to all of $T_{\infty}$.
\end{defin}

Regarding equation~\eqref{eq:Pdefmod},
note that every $\tau\in\Aut(T_n)$ has infinitely many extensions
$\tilde{\tau}\in\Aut(T_{\infty})$, since we may choose the parity
$\Par(\tilde{\tau},y)$ at each node $y$ at levels $n+1$ and higher
to be either $0$ or $1$ as we please. However, the definition
of $P(\tau,x)$ in equation~\eqref{eq:Pdefmod} is independent of the
extension $\tilde{\tau}$, since the contributions
of $\Par(\tilde{\tau},y)$ for nodes $y$ at levels $n+1$ and higher
from equations~\eqref{eq:Qdef} and~\eqref{eq:Pdef} all have coefficients
divisible by $2^j$. That is, when computing $P(\tau,x)$, we may simply
truncate the sums in equations~\eqref{eq:Qdef} and~\eqref{eq:Pdef}
to include only the contributions from nodes at levels $n-1$ and below.

It is immediate from equation~\eqref{eq:Qdef} that
\begin{equation}
\label{eq:Qnext}
Q(\sigma,x) = \Par(\sigma,x) + 2\sum_{s\in\{a,b\}}Q(\sigma,xas) ,
\end{equation}
where we understand this equation to be an equality in $\ZZ_2^{\times}$
in the $T_{\infty}$ case, and a congruence modulo an appropriate power of $2$
in the $T_n$ case.

To help explain Definition~\ref{def:Pmap}, observe that
$P(\sigma,x)$ is $\pm 1$ plus a weighted sum of $\Par(\sigma,y)$
at certain nodes $y$, chosen based on the labeling of the tree.
For example, Figure~\ref{fig:sgn1level5} shows the nodes in question up to level $5$.
To compute $P=P(\sigma,x)$ in that case, we count the highlighted nodes as follows:
\begin{itemize}
\item count gray circles $y$ for which $\Par(\sigma,y)=1$ with weight $-2$,
\item count white circles $y$ for which $\Par(\sigma,y)=1$ with weight $2$,
\item count gray squares $y$ for which $\Par(\sigma,y)=1$ with weight $-4$,
\item count white squares $y$ for which $\Par(\sigma,y)=1$ with weight $4$,
\end{itemize}
and so on up the tree.
Finally, add $1$ if $\Par(\sigma,x)=0$
or add $-1$ if $\Par(\sigma,x)=1$.

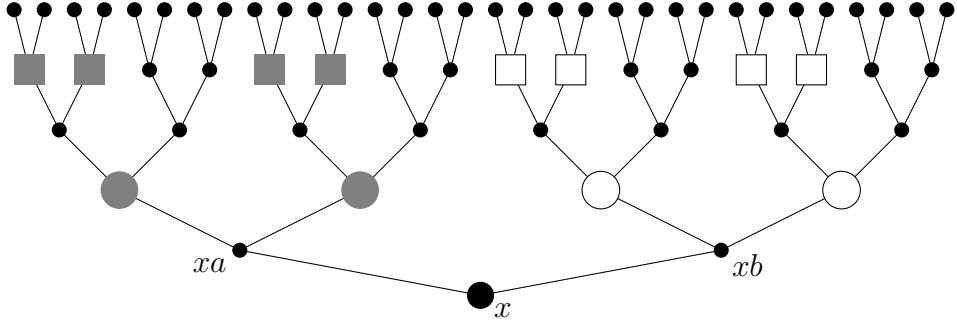
\begin{figure}
\begin{tikzpicture}
\path[draw] (0.2,4) -- (0.4,3.2) -- (0.6,4);
\path[draw] (1,4) -- (1.2,3.2) -- (1.4,4);
\path[draw] (1.8,4) -- (2,3.2) -- (2.2,4);
\path[draw] (2.6,4) -- (2.8,3.2) -- (3,4);
\path[draw] (3.4,4) -- (3.6,3.2) -- (3.8,4);
\path[draw] (4.2,4) -- (4.4,3.2) -- (4.6,4);
\path[draw] (5,4) -- (5.2,3.2) -- (5.4,4);
\path[draw] (5.8,4) -- (6,3.2) -- (6.2,4);
\path[draw] (6.6,4) -- (6.8,3.2) -- (7,4);
\path[draw] (7.4,4) -- (7.6,3.2) -- (7.8,4);
\path[draw] (8.2,4) -- (8.4,3.2) -- (8.6,4);
\path[draw] (9,4) -- (9.2,3.2) -- (9.4,4);
\path[draw] (9.8,4) -- (10,3.2) -- (10.2,4);
\path[draw] (10.6,4) -- (10.8,3.2) -- (11,4);
\path[draw] (11.4,4) -- (11.6,3.2) -- (11.8,4);
\path[draw] (12.2,4) -- (12.4,3.2) -- (12.6,4);
\path[fill] (0.2,4) circle (0.1);
\path[fill] (0.6,4) circle (0.1);
\path[fill] (1,4) circle (0.1);
\path[fill] (1.4,4) circle (0.1);
\path[fill] (1.8,4) circle (0.1);
\path[fill] (2.2,4) circle (0.1);
\path[fill] (2.6,4) circle (0.1);
\path[fill] (3,4) circle (0.1);
\path[fill] (3.4,4) circle (0.1);
\path[fill] (3.8,4) circle (0.1);
\path[fill] (4.2,4) circle (0.1);
\path[fill] (4.6,4) circle (0.1);
\path[fill] (5,4) circle (0.1);
\path[fill] (5.4,4) circle (0.1);
\path[fill] (5.8,4) circle (0.1);
\path[fill] (6.2,4) circle (0.1);
\path[fill] (6.6,4) circle (0.1);
\path[fill] (7,4) circle (0.1);
\path[fill] (7.4,4) circle (0.1);
\path[fill] (7.8,4) circle (0.1);
\path[fill] (8.2,4) circle (0.1);
\path[fill] (8.6,4) circle (0.1);
\path[fill] (9,4) circle (0.1);
\path[fill] (9.4,4) circle (0.1);
\path[fill] (9.8,4) circle (0.1);
\path[fill] (10.2,4) circle (0.1);
\path[fill] (10.6,4) circle (0.1);
\path[fill] (11,4) circle (0.1);
\path[fill] (11.4,4) circle (0.1);
\path[fill] (11.8,4) circle (0.1);
\path[fill] (12.2,4) circle (0.1);
\path[fill] (12.6,4) circle (0.1);
\path[draw] (0.4,3.2) -- (0.8,2.4) -- (1.2,3.2);
\path[draw] (2,3.2) -- (2.4,2.4) -- (2.8,3.2);
\path[draw] (3.6,3.2) -- (4,2.4) -- (4.4,3.2);
\path[draw] (5.2,3.2) -- (5.6,2.4) -- (6,3.2);
\path[draw] (6.8,3.2) -- (7.2,2.4) -- (7.6,3.2);
\path[draw] (8.4,3.2) -- (8.8,2.4) -- (9.2,3.2);
\path[draw] (10,3.2) -- (10.4,2.4) -- (10.8,3.2);
\path[draw] (11.6,3.2) -- (12,2.4) -- (12.4,3.2);
\path[fill,gray] (0.2,3.0) -- (0.6,3.0) -- (0.6,3.4) -- (0.2,3.4);
\path[fill,gray] (1.0,3.0) -- (1.4,3.0) -- (1.4,3.4) -- (1.0,3.4);
\path[fill] (2.0,3.2) circle (0.1);
\path[fill] (2.8,3.2) circle (0.1);
\path[fill,gray] (3.4,3.0) -- (3.8,3.0) -- (3.8,3.4) -- (3.4,3.4);
\path[fill,gray] (4.2,3.0) -- (4.6,3.0) -- (4.6,3.4) -- (4.2,3.4);
\path[fill] (5.2,3.2) circle (0.1);
\path[fill] (6,3.2) circle (0.1);
\path[fill,white] (6.6,3.0) -- (7,3.0) -- (7,3.4) -- (6.6,3.4);
\path[draw] (6.6,3.0) -- (7,3.0) -- (7,3.4) -- (6.6,3.4) -- (6.6,3.0);
\path[fill,white] (7.4,3.0) -- (7.8,3.0) -- (7.8,3.4) -- (7.4,3.4);
\path[draw] (7.4,3.0) -- (7.8,3.0) -- (7.8,3.4) -- (7.4,3.4) -- (7.4,3.0);
\path[fill] (8.4,3.2) circle (0.1);
\path[fill] (9.2,3.2) circle (0.1);
\path[fill,white] (9.8,3.0) -- (10.2,3.0) -- (10.2,3.4) -- (9.8,3.4);
\path[draw] (9.8,3.0) -- (10.2,3.0) -- (10.2,3.4) -- (9.8,3.4) -- (9.8,3.0);
\path[fill,white] (10.6,3.0) -- (11,3.0) -- (11,3.4) -- (10.6,3.4);
\path[draw] (10.6,3.0) -- (11,3.0) -- (11,3.4) -- (10.6,3.4) -- (10.6,3.0);
\path[fill] (11.6,3.2) circle (0.1);
\path[fill] (12.4,3.2) circle (0.1);
\path[draw] (0.8,2.4) -- (1.6,1.6) -- (2.4,2.4);
\path[draw] (4,2.4) -- (4.8,1.6) -- (5.6,2.4);
\path[draw] (7.2,2.4) -- (8,1.6) -- (8.8,2.4);
\path[draw] (10.4,2.4) -- (11.2,1.6) -- (12,2.4);
\path[fill] (0.8,2.4) circle (0.1);
\path[fill] (2.4,2.4) circle (0.1);
\path[fill] (4,2.4) circle (0.1);
\path[fill] (5.6,2.4) circle (0.1);
\path[fill] (7.2,2.4) circle (0.1);
\path[fill] (8.8,2.4) circle (0.1);
\path[fill] (10.4,2.4) circle (0.1);
\path[fill] (12,2.4) circle (0.1);
\path[draw] (1.6,1.6) -- (3.2,0.8) -- (4.8,1.6);
\path[draw] (8,1.6) -- (9.6,0.8) -- (11.2,1.6);
\path[fill,gray] (1.6,1.6) circle (0.25);
\path[fill,gray] (4.8,1.6) circle (0.25);
\path[fill,white] (8,1.6) circle (0.25);
\path[draw] (8,1.6) circle (0.25);
\path[fill,white] (11.2,1.6) circle (0.25);
\path[draw] (11.2,1.6) circle (0.25);
\path[draw] (3.2,0.8) -- (6.4,0.2) -- (9.6,0.8);
\path[fill] (3.2,0.8) circle (0.1);
\path[fill] (9.6,0.8) circle (0.1);
\path[fill] (6.4,0.2) circle (0.18);
\node (x) at (6.7,0) {$x$};
\node (xa) at (2.8,0.6) {$xa$};
\node (xb) at (9.95,0.6) {$xb$};
\end{tikzpicture}
\caption{$P(\sigma,x)$ is a weighted sum of  $\Par(\sigma,y)$
at the highlighted nodes $y$.}
\label{fig:sgn1level5}
\end{figure}
Note that the nodes highlighted in Figure~\ref{fig:sgn1level5},
and hence figuring into the sum defining $P$, 
are the \emph{parents} of the nodes highlighted in
Figure~\ref{fig:ablevel5} from Lemma~\ref{lem:nrel}.
Indeed, we now prove that $P(\sigma,x)$ is exactly
the requisite power suggested just before Definition~\ref{def:Pmap}.

\begin{thm}
\label{thm:Pembed}
Let $x_0\in K$ with $x_0\neq 0,-1$.
Choose a sequence $\{ \zeta_2, \zeta_4,\zeta_8,\ldots\}$
of primitive $2$-power roots of unity in $K_{\infty}$,
with $(\zeta_{2^{m}})^2 = \zeta_{2^{m-1}}$ for each $m\geq 1$.
Label the tree $T_{\infty}$ of preimages $\Orb_f^-(x_0)$
as in Lemma~\ref{lem:pickzeta}.
Then for any node $y\in\Orb_f^-(x_0)$,
any $2$-power root of unity $\zeta\in K_{\infty}$,
and any $\sigma\in G_{\infty}=\Gal(K_{\infty}/K)$, we have
\[ \sigma(\zeta) = \zeta^{P(\sigma,y)}. \]
\end{thm}

\begin{proof}
It suffices to show that $\sigma(\zeta_{2^m}) = \zeta_{2^m}^{P(\sigma,y)}$
for every $\sigma\in G_{\infty}$, every $m\geq 1$,
and every $y\in \Orb_f^-(x_0)$.
Throughout the proof, then, fix such $\sigma$, $m$, and $y$.

Because of the choice of labeling of the tree specified in the hypotheses,
for any point $w\in\Orb_f^-(x_0)$ and any $i\geq 0$,
we have
\begin{equation}
\label{eq:wswitch}
\prod_{t_1,\ldots, t_{i}} \big[ \sigma(wa)t_1 at_2 \cdots at_{i}a \big]
= \zeta_{2^{i+1}}^{-\Par(\sigma,w)}
\prod_{t_1,\ldots, t_{i}} \big[ \sigma(w)at_1 at_2 \cdots at_{i}a \big] .
\end{equation}
Indeed, if $\Par(\sigma,w)=1$, then equation~\eqref{eq:wswitch}
is simply equation~\eqref{eq:zetaprod} of Lemma~\ref{lem:pickzeta} applied to $\sigma(w)$;
and if $\Par(\sigma,w)=0$, then \eqref{eq:wswitch} is a tautology.
Each product in equation~\eqref{eq:wswitch} is over $t_1,\ldots,t_{i}\in\{a,b\}$;
and for $i=0$, we understand it to say $[\sigma(wa)] = \zeta_2^{-\Par(\sigma,w)} [\sigma(w)a]$.

In addition, since the two-element sets $\{\sigma(wa),\sigma(wb)\}$
and $\{\sigma(w)a,\sigma(w)b\}$ always coincide, we have
\begin{equation}
\label{eq:orderswap}
\prod_{s\in\{a,b\}} \big[ \sigma(w_1 s) w_2 \big] =
\prod_{s\in\{a,b\}} \big[ \sigma(w_1) s w_2 \big].
\end{equation}
for any strings $w_1$ and $w_2$ of the symbols $a,b$.
Thus, if we define
\[ P_j(\sigma,w):=\sum_{t_1,\ldots,t_j\in\{a,b\}} \Par(\sigma, wt_1at_2a\cdots at_j),\]
then for any node $x$ of the tree and any $m\geq 1$, we have
\begin{align}
\label{eq:bigprod}
\prod \big[\sigma(xs_1as_2\cdots a s_{m-1} a)\big]
&= \zeta_2^{-P_{m-1}(\sigma,x)} \prod \big[\sigma(xs_1as_2\cdots a s_{m-1}) a \big]
\notag \\
&= \zeta_2^{-P_{m-1}(\sigma,x)} \prod \big[\sigma(xs_1as_2\cdots a) s_{m-1} a \big]
\notag \\
&= \zeta_2^{-P_{m-1} (\sigma,x)} \zeta_4^{-P_{m-2}(\sigma,x)}
\prod \big[ \sigma(xs_1as_2\cdots s_{m-2} ) a s_{m-1} a \big]
\\
&= \cdots \notag \\
&=\bigg( \prod_{j=1}^{m-1} \zeta_{2^{m-j}}^{-P_{j}(\sigma,x)} \bigg)
\prod \big[ \sigma(x)s_1as_2\cdots a s_{m-1} a \big],
\notag
\end{align}
where each undecorated product is over $s_1,\ldots,s_{m-1}\in\{a,b\}$.
In proving equation~\eqref{eq:bigprod}, we have alternately 
applied equations~\eqref{eq:wswitch} and~\eqref{eq:orderswap}.
Specifically, we used equation~\eqref{eq:wswitch}
with $i=0$ and $w=xs_1a\cdots a s_{m-1}$ at the first equality,
then with $i=1$ and $w=xs_1a\cdots a s_{m-2}$ at the third,
and so on through $i=m-2$ with $w=x$.
We then used \eqref{eq:orderswap} at the second equality
with $w_1=xs_1a\cdots as_{m-2}a$ and $w_2=a$,
then with $w_1=xs_1a\cdots as_{m-3}a$ and $w_2=as_{m-1}a$ at the fourth,
and so on.

Applying equation~\eqref{eq:bigprod} to both $x=ya$ and $x=yb$,
and substituting the results in equation~\eqref{eq:zetaprod} with $i=m-1$,
we obtain
\begin{align}
\label{eq:sigmazeta}
\sigma(\zeta_{2^m}) &= 
\frac{\prod \big[ \sigma(yas_1 as_2 \ldots as_{m-1}a) \big]}
{\prod \big[ \sigma(ybs_1 as_2 \ldots as_{m-1}a) \big] }
\notag \\
& = \prod_{j=1}^{m-1} \zeta_{2^{m-j}}^{P_{j}(\sigma,yb) - P_{j}(\sigma,ya)}
\cdot \frac{\prod \big[ \sigma(ya)s_1 as_2 \ldots as_{m-1}a \big]}
{\prod \big[ \sigma(yb)s_1 as_2 \ldots as_{m-1}a \big] } ,
\end{align}
where each undecorated product is again over $s_1,\ldots,s_{m-1}\in\{a,b\}$.
Since $\zeta_{2^i}^2=\zeta_{2^{i-1}}$ for each $i$, the first product in
expression~\eqref{eq:sigmazeta} is $\zeta_{2^m}^M$, where
\begin{align*}
M &:= \sum_{j=1}^{m-1} 2^{j} \big(P_{j}(\sigma,yb) - P_{j}(\sigma,ya) \big)
\\
&\equiv 2\sum_{t\in\{a,b\}} Q(\sigma,ybt)
- 2\sum_{t\in\{a,b\}} Q(\sigma,yat) \pmod{2^m} ,
\end{align*}
and where $Q(\sigma,x)$ is as defined in equation~\eqref{eq:Qdef}.
On the other hand, by equation~\eqref{eq:zetaprod} applied to $\sigma(y)$,
the quotient of two products in expression~\eqref{eq:sigmazeta} is
$\zeta_{2^m}$ if $\Par(\sigma,y)=0$, or $\zeta_{2^m}^{-1} $ if $\Par(\sigma,y)=1$.
Thus, equation~\eqref{eq:sigmazeta} becomes
\[ \sigma(\zeta_{2^m}) = \zeta_{2^m}^P,
\quad \text{where} \quad P=(-1)^{\Par(\sigma,y)} +M = P(\sigma,y) .
\qedhere \]
\end{proof}

\section{The arithmetic basilica group}
\label{sec:Mn}
Theorem~\ref{thm:Pembed} says that $\sigma\in G_n$ or $\sigma\in G_{\infty}$
maps a $2$-power root of unity $\zeta$ to $\zeta^{P(\sigma,x)}$.
Since this image should be independent of the node $x$, we make the following definition.

\begin{defin}
\label{def:Mn}
Fix a labeling of $T_{\infty}$.
Let $x_0$ denote the root point of the tree.
Define the \emph{arithmetic basilica group}
$M_{\infty}$ to be the set of all $\sigma\in \Aut(T_{\infty})$
for which 
\[ P(\sigma,x) = P(\sigma,x_0) \quad\text{for every node } x\in T_{\infty} . \]
Similarly, given $n\geq 1$ and a labeling of $T_n$,
define $M_n$ to be the set of all $\sigma\in \Aut(T_n)$
for which the following condition holds:
for every $m\geq 0$, we have
\begin{equation}
\label{eq:Mncong}
P(\sigma,x) \equiv P(\sigma,x_0) \pmod{2^j}
\quad\text{for every node } x\in\{a,b\}^m ,
\end{equation}
where $j:= \lfloor (n-m+1)/2 \rfloor$.
\end{defin}

In light of Definition~\ref{def:Mn}, for $\sigma\in M_{\infty}$ or $\sigma\in M_n$
we may write simply $P(\sigma)$ instead of $P(\sigma,x)$,
since the value is independent of the node. That is,
we have functions $P:M_{\infty}\to \ZZ_2^{\times}$
and $P:M_n \to (\ZZ/2^j\ZZ)^{\times}$, where $j:=\lfloor (n+1)/2 \rfloor$.

\begin{thm}
\label{thm:Mgroup}
\hspace{2mm}
\begin{enumerate}
\item
$M_{\infty}$ is a subgroup of $\Aut(T_{\infty})$.
\item For every $n\geq 1$, $M_n$ is a subgroup of $\Aut(T_n)$.
\item
The map $P:M_{\infty}\to\ZZ_2^{\times}$ is a group homomorphism.
\item For every $n\geq 1$, the map $P:M_{n}\to(\ZZ/2^j\ZZ)^{\times}$
where $j:=\lfloor (n+1)/2 \rfloor$, is a group homomorphism.
\end{enumerate}
\end{thm}

\begin{proof}
We prove statements (1) and (3); the proofs of statements (2) and (4) are similar.
Our argument has five steps. In Step~1, we derive some simple formulas to be used
later in the proof. In Step~2, we define a function $Z_{\sigma,\tau}$ on the nodes
of the tree and prove that it satisfies a certain functional equation.
In Step~3, we apply this functional equation to show that $Z_{\sigma,\tau}$
is in fact identically zero, yielding the key identity~\eqref{eq:Qident}.
In Step~4, we use~\eqref{eq:Qident} to deduce identity~\eqref{eq:Phom},
which \emph{a posteriori} says that $P$ is a homomorphism. Finally, in Step~5 we use
identity~\eqref{eq:Phom} to show that $M_{\infty}$ is a subgroup of $\Aut(T_{\infty})$.

\medskip

\textbf{Step 1}. For any $\sigma\in\Aut(T_{\infty})$ and any node $x$ of $T_{\infty}$,
define
\begin{equation}
\label{eq:sgn1def}
\sgn(\sigma,x) := (-1)^{\Par(\sigma,x)} = 1-2\Par(\sigma,x).
\end{equation}
It is immediate from equation~\eqref{eq:sgndef} that
for any $\sigma,\tau\in\Aut(T_{\infty})$ and any node $x$ of $T_{\infty}$,
we have
\begin{equation}
\label{eq:sgn1}
\sgn(\sigma \tau, x) = \sgn\big(\sigma, \tau(x) \big) \cdot \sgn(\tau, x) .
\end{equation}
We also have
\begin{equation}
\label{eq:sgn2}
\Par(\sigma\tau,x) = \Par\big(\sigma, \tau(x)\big) + \sgn\big(\sigma,\tau(x)\big) \Par(\tau,x) .
\end{equation}
Equation~\eqref{eq:sgn2} follows from equation~\eqref{eq:sgn1} by writing
$\Par(\cdot,\cdot)=(1-\sgn(\cdot,\cdot))/2$, or simply by checking the four
possible choices of $\Par(\tau,x)$ and $\Par(\sigma,\tau(x))$.

\medskip

\textbf{Step 2}. For any $\sigma\in M_{\infty}$,
any $\tau\in\Aut(T_{\infty})$, and any node $x$ of $T_{\infty}$,
define
\[ Z_{\sigma,\tau}(x) :=
Q\big(\sigma, \tau(x) \big) + P(\sigma) Q(\tau,x) - Q(\sigma\tau,x),\]
where $P(\sigma)$ is the constant value of $P(\sigma,w)$ for all nodes $w$ of $T_{\infty}$.
In Step~3 we will show that $Z_{\sigma,\tau}$ is identically zero,
but for now we make the weaker claim that
\[ Z_{\sigma,\tau}(x) = 2Z_{\sigma,\tau}(xaa) + 2Z_{\sigma,\tau}(xab). \]
To see this, expand each appearance of $Q$ in the definition of
$Z_{\sigma,\tau}(x)$ according to equation~\eqref{eq:Qnext}, yielding
\begin{align}
Z_{\sigma,\tau}(x)
&= \Par\big(\sigma,\tau(x)\big) + P(\sigma) \Par(\tau, x)  - \Par(\sigma\tau,x)
\notag \\
&\qquad+ 2\sum_{s\in\{a,b\}} \Big[
Q\big(\sigma,\tau(x)as\big) + P(\sigma) Q(\tau,xas) - Q(\sigma\tau,xas) \Big]
\notag \\
&= \Par\big(\sigma,\tau(x)\big) + \sgn\big(\sigma,\tau(x)\big) \Par(\tau, x)  - \Par(\sigma\tau,x)
\label{eq:ZZ1} \\
&\qquad + 2\Par(\tau,x) \sum_{s\in\{a,b\}}
\Big[ Q\big(\sigma,\tau(x)bs\big) - Q\big(\sigma,\tau(x)as\big) \Big]
\label{eq:ZZ2} \\
&\qquad+ 2\sum_{s\in\{a,b\}} \Big[
Q\big(\sigma,\tau(x)as\big) + P(\sigma) Q(\tau,xas) - Q(\sigma\tau,xas) \Big] ,
\label{eq:ZZ3}
\end{align}
where in the second equality, we expanded the first appearance of $P(\sigma)$
as $P(\sigma,\tau(x))$. The expression on line~\eqref{eq:ZZ1} is zero
by equation~\eqref{eq:sgn2}. Next, observe that
\begin{equation}
\label{eq:tausign}
\big\{ \tau(x)aa, \tau(x)ab \big\} = 
\begin{cases}
\dsps \big\{ \tau(xaa), \tau(xab) \big\} & \text{ if } \Par(\tau,x)=0,
\\
\dsps \big\{ \tau(xba), \tau(xbb) \big\} & \text{ if } \Par(\tau,x)=1,
\end{cases}
\end{equation}
and similarly for the set $\{ \tau(x)ba, \tau(x)bb \}$.
Thus, the expression on line~\eqref{eq:ZZ2} is
\[ \begin{cases}
\dsps 0 & \text{ if } \Par(\tau,x)=0,
\\
\dsps 2 \sum_{s\in\{a,b\}}
\Big[ Q\big(\sigma,\tau(xas)\big) - Q\big(\sigma,\tau(xbs)\big) \Big]
& \text{ if } \Par(\tau,x)=1.
\end{cases} \]
and the expression on line~\eqref{eq:ZZ3} is
\[ \begin{cases}
\dsps 2\sum_{s\in\{a,b\}} \Big[
Q\big(\sigma,\tau(xas)\big) + P(\sigma) Q(\tau,xas) - Q(\sigma\tau,xas) \Big] ,
& \text{ if } \Par(\tau,x)=0,
\\
\dsps 2\sum_{s\in\{a,b\}} \Big[
Q\big(\sigma,\tau(xbs)\big) + P(\sigma) Q(\tau,xas) - Q(\sigma\tau,xas) \Big] ,
& \text{ if } \Par(\tau,x)=1.
\end{cases} \]
For either possible value of $\Par(\tau,x)$, then, we have
\begin{align*}
Z_{\sigma,\tau}(x)
&= 2\sum_{s\in\{a,b\}} \Big[
Q\big(\sigma,\tau(xas)\big) + P(\sigma) Q(\tau,xas) - Q(\sigma\tau,xas) \Big] 
\\
&= 2Z_{\sigma,\tau}(xaa) + 2Z_{\sigma,\tau}(xab),
\end{align*}
proving our claim.

\medskip

\textbf{Step 3}.
As in Step~2, consider $\sigma\in M_{\infty}$ and $\tau\in\Aut(T_{\infty})$.
We claim that $Z_{\sigma,\tau}$ is identically zero; that is, we claim
\begin{equation}
\label{eq:Qident}
Q\big(\sigma, \tau(x) \big) + P(\sigma) Q(\tau,x) = Q(\sigma\tau,x)
\end{equation}
for every node $x$ of $T_{\infty}$.
To prove this, it suffices to show that for every node $x$ and each $j\geq 0$,
we have $Z_{\sigma,\tau}(x)\in 2^j \ZZ_2$.

We proceed by induction on $j$.
The base case $j=0$ is immediate from the fact that
$P(\cdot,\cdot), Q(\cdot,\cdot)\in\ZZ_2$.
Assuming the statement holds for all $x$
for some particular $j\geq 0$, then for any node $x$,
Step~2 yields
\[  Z_{\sigma,\tau}(x)
= 2\big(Z_{\sigma,\tau}(xaa) + Z_{\sigma,\tau}(xab)\big)
\in 2\big( 2^j \ZZ_2 \big) = 2^{j+1} \ZZ_2, \]
completing the induction and proving our claim.

\medskip

\textbf{Step 4}.
As in the previous two steps, consider $\sigma\in M_{\infty}$ and $\tau\in\Aut(T_{\infty})$,
and consider a node $x$ in $T_{\infty}$. We claim that
\begin{equation}
\label{eq:Phom}
P(\sigma) P(\tau, x) = P(\sigma\tau, x).
\end{equation}
Indeed, expanding $P(\tau,x)$ yields
\begin{align*}
P(\sigma) P(\tau, x) &=
\sgn(\tau,x)P(\sigma) + 2 \sum_{t\in \{a,b\}} P(\sigma)Q(\tau,xbt)
- 2 \sum_{t\in \{a,b\}} P(\sigma)Q(\tau,xat)
\\
&= \sgn(\tau,x) \bigg[ \sgn\big(\sigma,\tau(x)\big) 
+ 2 \sum_{t\in \{a,b\}} Q\big(\sigma,\tau(x)bt\big)
- 2 \sum_{t\in \{a,b\}} Q\big(\sigma,\tau(x)at\big) \bigg]
\\
& \qquad
+ 2 \sum_{t\in \{a,b\}} P(\sigma)Q(\tau,xbt)
- 2 \sum_{t\in \{a,b\}} P(\sigma)Q(\tau,xat),
\end{align*}
where we have also expanded the first appearance of $P(\sigma)$
as $P(\sigma,\tau(x))$. Applying equations~\eqref{eq:sgn1} and~\eqref{eq:tausign}, then, we have
\begin{align*}
P(\sigma) P(\tau, x) &=
\sgn(\sigma\tau,x)
+ 2 \sum_{t\in \{a,b\}} \Big[ Q\big(\sigma,\tau(xbt)\big) + P(\sigma)Q(\tau,xbt) \Big]
\\
& \qquad
- 2 \sum_{t\in \{a,b\}} \Big[ Q\big(\sigma,\tau(xat)\big) + P(\sigma)Q(\tau,xat) \Big]
\\
&= \sgn(\sigma\tau,x)
+ 2\sum_{t\in \{a,b\}} Q(\sigma \tau, xbt) - 2\sum_{t\in \{a,b\}} Q(\sigma \tau, xat)
= P(\sigma \tau,x),
\end{align*}
where we used identity~\eqref{eq:Qident} twice in the second equality, thus proving our claim.

\medskip

\textbf{Step 5}.
To prove statement~(1), first observe that the identity automorphism $e\in\Aut(T_{\infty})$
belongs to $M_{\infty}$, since $\Par(e,y)=0$ for all nodes $y$, and hence
$P(e,x)=1$ for all nodes $x$ of $T_{\infty}$.
Next, given $\sigma,\tau\in M_{\infty}$, it follows from identity~\eqref{eq:Phom}
that for any node $x$ of $T_{\infty}$, we have
\[ P(\sigma\tau,x) = P(\sigma)P(\tau,x)
= P(\sigma)P(\tau,x_0) = P(\sigma\tau,x_0), \]
and hence $\sigma\tau\in M_{\infty}$.
Finally, given $\sigma\in M_{\infty}$, consider $\sigma^{-1}\in\Aut(T_{\infty})$.
Then for any node $x$ of $T_{\infty}$, identity~\eqref{eq:Phom} again yields
\[ 1 = P(e,x) = P(\sigma\sigma^{-1},x) = P(\sigma)P(\sigma^{-1},x). \]
Thus,
\[ P\big(\sigma^{-1},x\big)=P(\sigma)^{-1} = P\big(\sigma^{-1},x_0\big), \]
and therefore $\sigma^{-1}\in M_{\infty}$.
That is, $M_{\infty}$ is indeed a subgroup of $\Aut(T_{\infty})$.

Finally, the fact that $P:M_{\infty}\to\ZZ_2^{\times}$ is a homomorphism
is immediate from identity~\eqref{eq:Phom},
proving statement~(2).
\end{proof}

Since $M_{\infty}$ is indeed a group by Theorem~\ref{thm:Mgroup},
the following more precise version of statement~(1) of our Main Theorem
is an immediate consequence of Theorem~\ref{thm:Pembed}.

\begin{cor}
\label{cor:Pbasil}
Fix notation and a tree labeling as in Theorem~\ref{thm:Pembed}. Consider
the embedding of $G_{\infty}$ in $\Aut(T_{\infty})$
induced by its action on $\Orb_f^-(x_0)$.
Then the image of this embedding is contained in
the arithmetic basilica group $M_{\infty}$.
\end{cor}

\section{The basilica group and finite subtrees}
\label{sec:Basil}

For any $0\leq m\leq n\leq \infty$, define a function $R_{n,m}:\Aut(T_n)\to\Aut(T_m)$
by restricting $\sigma\in\Aut(T_n)$ to the subtree $T_m$.
Clearly $R_{n,m}$ is a homomorphism.

The group $\Aut(T_\infty)$ has a topological structure, as follows. For each $m\geq 0$,
let $U_m:=\ker(R_{\infty,m})$.
That is, $U_m$ consists of all $\sigma\in\Aut(T_{\infty})$
that act trivially on levels $0$ through $m$ of the tree.
The cosets of the normal subgroups $U_m$ form a basis for a topology
on $\Aut(T_{\infty})$, making $\Aut(T_{\infty})$ compact and Hausdorff.
For $n\geq m\geq 0$,
we will often abuse notation and write $U_m$ for the subgroup
$R_{\infty,n}(U_m) = \ker(R_{n,m})$ of $\Aut(T_n)$.

For any node $x$ of $T_{\infty}$, it is immediate from Definition~\ref{def:Pmap}
that $\sigma\mapsto P(\sigma,x)$ is a continuous function from
$\Aut(T_\infty)$ to $\ZZ_2^{\times}$. It follows that $M_{\infty}$
is a closed and hence compact subgroup of $\Aut(T_{\infty})$.

Fix a labeling of the tree $T_{\infty}$.
Define two particular automorphisms $\mapa,\mapb \in\Aut(T_{\infty})$ by specifying that
\begin{itemize}
\item $\Par(\mapa,bb\cdots b) = 1$ for any node whose label is a string
of an even number of $b$'s,
\item $\Par(\mapa,y)=0$ for all other nodes $y$ of $T_{\infty}$,
\end{itemize}
and
\begin{itemize}
\item $\Par(\mapb,bb\cdots b) = 1$ for any node whose label is a string 
of an odd number of $b$'s,
\item $\Par(\mapb,y)=0$ for all other nodes $y$ of $T_{\infty}$.
\end{itemize}
The maps $\mapa$ and $\mapb$ can be equivalently defined by the recursive relations
\[ \mapa(aw) = b w, \quad \mapa(bw)= a\mapb(w), \qquad
\mapb(aw) = aw, \quad \mapb(bw)=b \mapa(w) \]
for any word $w$ in the symbols $a,b$.

\begin{defin}
\label{def:Bgroup}
The \emph{basilica group} is the subgroup $B_{\infty}$ of $\Aut(T_{\infty})$
generated by $\mapa$ and $\mapb$.
The \emph{closed basilica group} is the topological closure $\overline{B}_{\infty}$
of $B_{\infty}$ in $\Aut(T_{\infty})$.
\end{defin}

\begin{remark}
\label{rem:splitB}
Consider $\sigma\in U_1$, i.e., $\sigma\in\Aut(T_\infty)$ fixing level $1$
of the tree.
Then $\sigma$ acts on the subtree $T_{\infty,a}$ rooted at $a$
as some automorphism $\sigma_a\in\Aut(T_{\infty})$,
and similarly on the subtree $T_{\infty,b}$ rooted at $b$
as some $\sigma_b\in\Aut(T_{\infty})$.
That is, we may write $\sigma=(\sigma_a,\sigma_b)$.

In this notation, we have $\mapb=(e,\mapa)$, where $e$ is the identity element
of $\Aut(T_{\infty})$. Similarly, 
$\mapa^{-1}\mapb\mapa$ and $\mapa^2$ also belong to $U_1$,
and simple computations show that
$\mapa^{-1}\mapb\mapa = (\mapa, e)$ and $\mapa^2=(\mapb,\mapb)$.

Consider $\sigma\in B_{\infty}\cap U_1$. Then $\sigma$ must be a
finite product of $\mapa$ and $\mapb$ involving an even number of
copies of $\mapa$. (The parity condition on $\mapa$ is because $\sigma\in U_1$).
Any such product can also be written as a product of powers of
$\mapa^2$, $\mapb$, and $\mapa^{-1}\mapb\mapa$.
Thus, writing $\sigma=(\sigma_a,\sigma_b)$, we must have
$\sigma_a,\sigma_b\in B_{\infty}$.
Conversely, for any $\sigma_b\in B_{\infty}$, there is
some $\sigma_a\in B_{\infty}$ such that $(\sigma_a,\sigma_b)\in B_{\infty}\cap U_1$.
For this reason, the basilica group $B_{\infty}$ is said to be a \emph{self-similar group}.
See \cite{Nek} for more on self-similar groups,
especially Sections~3.10.2, 5.2.2, and 6.12.1, which specifically concern $B_{\infty}$.
\end{remark}

\begin{defin}
\label{def:BnEn}
Fix $n\geq 1$ and a labeling of $T_\infty$. Define
\begin{enumerate}
\item $B_n:=R_{\infty,n}(B_{\infty})=R_{\infty,n}(\overline{B}_{\infty})$.
\item $B'_n:=\ker(P:M_n\to(\ZZ/2^j\ZZ)^{\times})$,
where $j:=\lfloor (n+1)/2\rfloor$.
\item $E_n:=U_{n-1} \cap B'_n$.
\end{enumerate}
\end{defin}

Recall from Theorem~\ref{thm:Mgroup} that the map $P$ used to
define $B'_n$ above is indeed a homomorphism, so that both
$B_n$ and $B'_n$ are subgroups of $\Aut(T_n)$.
Moreover, a simple computation shows $P(\mapa,x)=P(\mapb,x)=1$
for every node $x$ of the tree $T_{\infty}$. Thus, we have $B_n\subseteq B'_n$.
In fact, these two groups coincide, as we will show in Theorem~\ref{thm:Bndef}.

To that end, consider the group
$E_{n-1}\times E_{n-1}$ acting on the tree $T_n$, where the first copy of $E_{n-1}$
acts on the copy of $T_{n-1}$ rooted at node $a$, and the second acts on the copy
of $T_{n-1}$ rooted at $b$. We have the following result.

\begin{lemma}
\label{lem:inductEn}
For every $n\geq 2$, $E_n$ is a subgroup of $E_{n-1}\times E_{n-1}$, and
\[ [E_{n-1}\times E_{n-1} : E_n ] = \begin{cases}
1 & \text{ if $n$ is even}, \\
2 & \text{ if $n$ is odd}.
\end{cases} \]
\end{lemma}

\begin{proof}
Every element $\sigma$ of either group $E_n$ or $E_{n-1}\times E_{n-1}$
acts trivially on the subtree $T_{n-1}$, with
\[ P(\sigma,x) \equiv 1 \pmod{2^i} \]
for every $1\leq m\leq n-1$ and every node $x\in\{a,b\}^m$,
where $i:= \lfloor (n-m+1)/2 \rfloor$.
This is the full set of defining conditions for $E_{n-1}\times E_{n-1}$,
but for $E_n$ there is the extra condition that
\begin{equation}
\label{eq:specialEn}
P(\sigma,x_0) \equiv 1 \pmod{2^j},
\quad \text{where  } j:= \lfloor (n+1)/2 \rfloor .
\end{equation}
Thus, $E_n=\{\sigma\in E_{n-1}\times E_{n-1} : P(\sigma,x_0) \equiv 1 \pmod{2^j} \}$
is clearly a subgroup of $E_{n-1}\times E_{n-1}$.
For $n$ even, equation~\eqref{eq:specialEn}
is true for all $\sigma\in U_{n-1}$ and hence all $\sigma\in E_{n-1}\times E_{n-1}$.
Thus, it suffices to show that the index is $2$ when $n$ is odd.

For the remainder of the proof, assume $n$ is odd, so that $n=2j-1$ with $j\geq 2$.
For $\sigma\in E_{n-1}\times E_{n-1} \subseteq U_{n-1}$,
we have $P(\sigma,x_0) \equiv 1 \pmod{2^{j-1}}$ by definition of $P$.
Since $E_{n-1}\times E_{n-1}$ is a subgroup of $M_n$
and $P:M_n\to(\ZZ/2^j\ZZ)^{\times}$ is a homomorphism, it suffices to show that
there is some $\lambda\in E_{n-1}\times E_{n-1}$ for which
$P(\lambda,x_0)\equiv 1+2^{j-1} \pmod{2^j}$.

Let $\beta:=R_{\infty,n-1}(\mapb)\in B'_{n-1}$,
and let $\lambda:=(e,\beta^{2^{(j-2)}})\in B'_{n-1} \times B'_{n-1}$.
Because $\mapb=(e,\mapa)$ and $\mapa^2=(\mapb,\mapb)$,
a simple induction on $j$ shows that that $\beta^{2^{(j-2)}}$ acts trivially on $T_{2j-3}=T_{n-2}$,
and hence $\lambda\in E_{n-1}\times E_{n-1}$.
Finally, another induction shows that for every node $y\in \{a,b\}^{2j-3}$, we have
\[ \Par\big(\beta^{2^{(j-2)}},y \big)
=\begin{cases}
1 & \text{ if } y = bs_1 b s_2 b \cdots s_{j-2} b \text{ for some } s_1,\ldots,s_{j-2}\in\{a,b\},
\\
0 & \text{ otherwise.}
\end{cases}
\]
Thus, when computing $Q(\lambda,st) \pmod{2^{j-1}}$ for any $s,t\in\{a,b\}$,
the only nontrivial term in the sum of equation~\eqref{eq:Qdef}
occurs for $s=t=b$, and it is $2^{j-2} \Par(\lambda,bbaba\cdots bab)$.
Therefore, Definition~\ref{def:Pmap} yields
\[ P(\lambda, x_0) \equiv (-1)^0 + 2\cdot 2^{j-2} -2\cdot 0 \equiv 1 + 2^{j-1} \pmod{2^j}, \]
as desired.
\end{proof}

\begin{thm}
\label{thm:Bndef}
For every $n\geq 1$,
we have $\dsps |E_n|=2^{e_n}$ and $\dsps |B_n| = |B'_n| =2^{b_n}$, where 
\begin{align}
\label{eq:bndef}
e_n &= \frac{2^n}{3} + \begin{cases}
2/3 & \text{ if $n$ is even,} \\
1/3 & \text{ if $n$ is odd,}
\end{cases}
\notag \\
b_n &= \frac{2^{n+1}}{3} + \frac{n}{2} - \begin{cases}
2/3 & \text{ if $n$ is even,} \\
5/6 & \text{ if $n$ is odd,}
\end{cases}
\end{align}
In particular, for every $n\geq 1$, we have $B_n=B'_n$.
Moreover, the closed basilica group $\overline{B}_{\infty}=\overline{\la \mapa, \mapb \ra}$
is precisely the kernel of the homomorphism $P:M_{\infty}\to\ZZ_2^{\times}$
\end{thm}

\begin{proof}
Since $E_1=B_1=B'_1=\Aut(T_1)$ is a group of order $2$, and since 
formulas~\eqref{eq:bndef} specify $e_1=b_1=1$, we have
$\dsps |E_1|=2^{e_1}$ and $\dsps |B_1| = |B'_1| =2^{b_1}$.
For all $n\geq 2$, the formula for $e_n$ clearly satisfies
\[ e_n = 2e_{n-1} - \begin{cases}
0 & \text{ if $n$ is even},  \\
1 & \text{ if $n$ is odd},
\end{cases} \]
and Lemma~\ref{lem:inductEn} says that the same relation holds for $\log_2 |E_n|$.
By induction, then, $\log_2 |E_n|=e_n$ for all $n\geq 1$.

Another simple induction shows that $b_n=e_1+e_2 + \cdots +e_n$.
However, we have not yet proven the equality of $b_n$, $\log_2 |B_n|$, and $\log_2 |B'_n|$
for $n\geq 2$.

Write $b'_n:=\log_2 |B'_n|$, and note that $b'_1=1$.
For $n\geq 2$, the homomorphism $R_{n,n-1}:B'_n\to B'_{n-1}$ has kernel $E_n$,
but we do not (yet) know that it is surjective.
Thus, $b'_n \leq b'_{n-1} + e_{n-1}$, and hence
\begin{equation}
\label{eq:bnineq}
b'_n \leq e_1 + \cdots + e_n = b_n.
\end{equation}
On the other hand, according to \cite[Proposition~2.3.1]{Pink}, we have
\begin{equation}
\label{eq:PinkBn}
\log_2 |B_n| = 2^n -1 - \sum_{m=0}^{n-1} 2^{n-1-m} \cdot \left\lfloor \frac{m}{2} \right\rfloor 
\end{equation}
Using the elementary identity
\[ \sum_{i=0}^{j-1} \frac{i}{4^i} = \frac{1}{4} \cdot \frac{(j-1) 4^{-j} - j 4^{1-j} + 1}{(3/4)^2}
= \frac{1}{9\cdot 4^{j-1}} \big( 4^j - 3j-1 \big), \]
we have, for $n=2j$ even,
\[ \sum_{m=0}^{n-1} 2^{n-1-m} \cdot \left\lfloor \frac{m}{2} \right\rfloor = \sum_{i=0}^{j-1} 3i \cdot 2^{2j-2i-2}
= \frac{1}{3} \big( 4^j - 3j-1 \big) = \frac{2^n}{3} - \frac{n}{2} - \frac{1}{3}, \]
and for $n=2j+1$ odd,
\[ \sum_{m=0}^{n-1} 2^{n-1-m} \cdot \left\lfloor \frac{m}{2} \right\rfloor = j + \sum_{i=0}^{j-1} 3i \cdot 2^{2j-2i-1}
= j+\frac{2}{3} \big( 4^j - 3j-1 \big) = \frac{2^n}{3} - \frac{n}{2} - \frac{1}{6}. \]
Substituting these expressions into equation~\eqref{eq:PinkBn} yields
$\log_2 |B_n| = b_n$, where $b_n$ is as in equation~\eqref{eq:bndef}.
By  inequality~\eqref{eq:bnineq} and the fact that $B_n\subseteq B'_n$, we have
\[ b'_n \leq b_n = \log_2 |B_n| \leq \log_2 |B'_n| = b'_n . \]
Thus, $\dsps |B_n| = |B'_n| =2^{b_n}$, and $B_n=B'_n$.

Finally, as noted near the start of this section,
the map $P:M_{\infty}\to\ZZ_2^{\times}$ is continuous.
Therefore, since $\mapa,\mapb\in\ker(P)$, we have $\overline{B}_{\infty}\subseteq\ker(P)$.
Conversely, given $\sigma\in\ker(P)$, define $\sigma_n:=R_{\infty,n}(\sigma)\in B'_n$
for each $n\geq 1$.
Because $\sigma\in B'_n=B_n$, 
there exists $\tau_n\in B_{\infty}$ such that $R_{\infty,n}(\tau_n)=\sigma_n$.
For any integer $n\geq 1$, the automorphisms $\sigma,\tau_n\in\Aut(T_{\infty})$
agree on $T_{n-1}$ and hence belong to the same coset of the subgroup $U_{n-1}$.
Thus, we have
\[ \sigma = \lim_{n\to\infty} \tau_n \in \overline{B}_{\infty}. \qedhere \]
\end{proof}

\begin{thm}
\label{thm:MBmap}
The homomorphism $P:M_{\infty}\to \ZZ_2^{\times}$ is surjective, and
we have the short exact sequence
\[ 0 \longrightarrow \overline{B}_{\infty} \longrightarrow M_{\infty}
\overset{P}{\longrightarrow} \ZZ_2^{\times} \longrightarrow 0 \]
In addition, for each $n\geq 1$, we have the short exact sequence
\[ 0 \longrightarrow B_n \longrightarrow M_n
\overset{P}{\longrightarrow} (\ZZ/2^j\ZZ)^{\times} \longrightarrow 0 , \]
where $j:= \lfloor (n+1)/2 \rfloor$. Moreover, $|M_n| = 2^{m_n}$, where
\begin{equation}
\label{eq:mndef}
m_n = \frac{2^{n+1}}{3} + n - \begin{cases}
5/3 & \text{ if $n$ is even,} \\
4/3 & \text{ if $n$ is odd.}
\end{cases}
\end{equation}
\end{thm}

\begin{proof}
\textbf{First statement}.
By Theorem~\ref{thm:Mgroup}, the map $P$ is a homomorphism,
and by Theorem~\ref{thm:Bndef}, its kernel is $\overline{B}_{\infty}$.
It remains to show that $P$ is surjective.

Fix a labeling of $T_{\infty}$. Define an automorphism $\mape \in\Aut(T_{\infty})$ by
\begin{equation}
\label{eq:epsdef}
\Par(\mape,x)=1 \text{ for all nodes $x$ of } T_{\infty} .
\end{equation}
It is immediate from Definition~\ref{def:Pmap} that
$Q(\mape,x)=1+4+4^2+\cdots = -1/3\in\ZZ_2$
for every node $x$, and hence that $P(\mape,x)=-1\in\ZZ_2$.
In particular, $\mape\in M_{\infty}$, with $P(\mape)=-1$.

Define another automorphism $\mapt\in\Aut(T_{\infty})$ inductively, as follows.
First, define
\[ \Par(\mapt,x_0):=\Par(\mapt,a):=\Par(\mapt,b):=0, \]
so that $\mapt$ acts trivially on $T_2$. Then,
once we have defined $\Par(\mapt,x)$ at a particular node $x$,
define $\Par(\mapt,y)$ for each node $y$ two levels above $x$ by:
\begin{align}
\label{eq:thetadef}
\Par(\mapt,xaa) & := \Par(\mapt,xab) :=0, \notag \\
\Par(\mapt,xba) &:=1, \quad \text{and} \\
\Par(\mapt,xbb) &:= \Par(\mapt,x). \notag
\end{align}
Because $\Par(\mapt,yaa) = \Par(\mapt,yab)=0$ for any node $y$,
we have $Q(\mapt,x)=\Par(\mapt,x)$ for all nodes $x$ of the tree.
Thus, according to Definition~\ref{def:Pmap} and equations~\eqref{eq:thetadef},
we have
\[ P(\mapt,x) = \begin{cases}
(-1)^0 + 2(1 + 0 - 0 - 0) = 3 & \text{ if } \Par(\mapt,x)=0,
\\
(-1)^1 + 2(1 + 1 - 0 - 0) = 3 & \text{ if } \Par(\mapt,x)=1.
\end{cases} \]
Therefore, $\mapt\in M_{\infty}$, with $P(\mapt)=3$.

Because of the automorphisms $\mape,\mapt\in M_{\infty}$,
the image of $P$ contains the closure
of the subgroup $\la -1, 3 \ra$ of $\ZZ_2^{\times}$ generated by $-1$ and $3$.
However, $\{-1,3\}$ is a set of topological generators for $\ZZ_2^{\times}$;
therefore, the image of $P$ is all of $\ZZ_2^{\times}$, as desired.

\textbf{Second statement}.
By Theorem~\ref{thm:Bndef}, we know that $B_n=B'_n$ is the
kernel of the homomorphism $P:M_n\to (\ZZ/2^j\ZZ)^{\times}$.
This homomorphism is surjective, since $-1$ and $3$ together generate
the group $(\ZZ/2^j\ZZ)^{\times}$, and they are the images of
$R_{\infty,n}(\mape)$ and $R_{\infty,n}(\mapt)$ under $P$.
Thus, we have the desired exact sequence.

\textbf{Third statement}.
By the second statement, we have $|M_n|=|(\ZZ/2^j\ZZ)^{\times}| \cdot |B_n|$
for every $n\geq 1$, where $j:=\lfloor (n+1)/2 \rfloor$. Therefore, by Theorem~\ref{thm:MBmap},
\begin{align*}
\log_2 |M_n| &= (j-1) + b_n
=  \left\lfloor \frac{n+1}{2} \right\rfloor -1 + \frac{2^{n+1}}{3} + \frac{n}{2} - \begin{cases}
2/3 & \text{ if $n$ is even,} \\
5/6 & \text{ if $n$ is odd,}
\end{cases}
\\
&= \frac{2^{n+1}}{3} + n - \begin{cases}
5/3 & \text{ if $n$ is even,} \\
4/3 & \text{ if $n$ is odd.}
\end{cases}
\qedhere
\end{align*}
\end{proof}

To help clarify formulas~\eqref{eq:bndef} and~\eqref{eq:mndef}, the following table gives
the logarithms of the orders of the 2-groups $E_n$, $B_n$, $M_n$, and $\Aut(T_n)$
for some small values of $n$.

\smallskip

\hfill
\begin{tabular}{|c|c|c|c|c|c|c|c|c|c|c|}
\hline
$n$ & 1 & 2 & 3 & 4 & 5 & 6 & 7 & 8 & 9 & 10
\\
\hline
$\log_2 |E_n|$ & $1$ & $2$ & $3$ & $6$ & $11$ & $22$ & $43$ & $86$
& $171$ & $342$
\\
\hline
$\log_2 |B_n|$ & $1$ & $3$ & $6$ & $12$ & $23$ & $45$ & $88$ & $174$
& $345$ & $687$
\\
\hline
$\log_2 |M_n|$ & $1$ & $3$ & $7$ & $13$ & $25$ & $47$ & $91$ & $177$
& $349$ & $691$
\\
\hline
$\log_2 |\Aut(T_n)|$ & $1$ & $3$ & $7$ & $15$ & $31$ & $63$ & $127$ & $255$
& $511$ & $1023$
\\
\hline
\end{tabular}
\hfill
{}

\section{The Arithmetic Basilica as a Galois Group}
\label{sec:mainproof}
In \cite{Pink}, Pink described $\Gal(K_{\infty}/K)$ in the case that $K$ is a function field.
We are now prepared to prove that Pink's description of $M_{\infty}$
--- in terms of generators and normalizers --- coincides with our description
in terms of the map $P$.

\begin{thm}
\label{thm:PinkEquiv}
Let $k$ be a field of characteristic different from $2$, and suppose that $[k(\zeta_8):k]=4$,
where $\zeta_8$ is a primitive eighth root of unity. Let $K=k(t)$, where $t$ is transcendental
over $k$, let $x_0:=t\in K$, and let $G_{\infty}:=\Gal(K_{\infty}/K)$.
Then $G_{\infty}\cong M_{\infty}$.
\end{thm}

%
\begin{proof}
Let $K':=\kbar(t)$, where $\kbar$ is an algebraic closure of $k$,
and let $K'_{\infty}=\bigcup_{n\geq 0} K'_n$, where $K'_n:=K'(f^{-n}(t))$.

Label the tree $T_{\infty}$ of preimages $\Orb_f^-(x_0)$
as in Lemma~\ref{lem:pickzeta}.
For compatibility with Pink's labeling in \cite{Pink},
identify our label $a$ with the label $1$ of \cite{Pink},
and identify our label $b$ with the label $0$ of \cite{Pink}.
Then in the notation of \cite[equation~(2.0.1)]{Pink},
our automorphism $\mapa$ is $a_1$, and our automorphism $\mapb$ is $a_2$.

By Corollary~\ref{cor:Pbasil}, there is a homomorphism
$G_{\infty}\hookrightarrow M_{\infty}$ respecting the action on the tree.
By \cite[Theorem~2.8.2]{Pink}, restricting this homomorphism yields
an isomorphism $\Gal(K'_{\infty}/K') \cong \overline{B}_{\infty}$.
In addition, by \cite[Theorem~2.8.4]{Pink}, the Galois group $\Gal(K'_{\infty}/K')$
fits into a short exact sequence
\begin{equation}
\label{eq:sesPink}
0 \longrightarrow \Gal(K'_{\infty}/K')  \longrightarrow G_{\infty} 
\overset{\rho}{\longrightarrow} \ZZ_2^{\times} \longrightarrow 0
\end{equation}
where $\rho: G_{\infty}\to \ZZ_2^{\times}$ is given by the cyclotomic action
of Galois on $2$-power roots of unity.
The sequence~\eqref{eq:sesPink} and that of Theorem~\ref{thm:MBmap}
together yield the diagram
\[ \begin{tikzcd}
0 \arrow[r]
  & \Gal(K'_{\infty}/K') \arrow[r] \arrow[d, "\wr"]
  & G_{\infty} \arrow[r, "\rho"] \arrow[d, hook]
  & \ZZ_2^{\times} \arrow[r] \arrow[d, equal]
  & 0
\\
0 \arrow[r]
  &  \overline{B}_{\infty} \arrow[r]
  & M_{\infty} \arrow[r, "P"]
  & \ZZ_2^{\times} \arrow[r] 
  & 0
\end{tikzcd} \]
The first square of the diagram commutes because the isomorphism
$\Gal(K'_{\infty}/K') \cong \overline{B}_{\infty}$ is the restriction of the injection
$G_{\infty}\hookrightarrow M_{\infty}$.
The second square also commutes, because
Theorem~\ref{thm:Pembed} says that $P$ is also given by the cyclotomic action
of Galois on $2$-power roots of unity.
Therefore, the homomorphism $G_{\infty}\hookrightarrow M_{\infty}$
is an isomorphism.
\end{proof}

With the goal of extending Theorem~\ref{thm:PinkEquiv} to the case that the base field $K$
is a number field,
recall that the \emph{Frattini subgroup} of a (profinite) group $G$ is the intersection $\Phi$
of all maximal (closed) subgroups of $G$. Clearly $\Phi$ is a normal subgroup of $G$.
We will need the following well-known fact,
which we prove for the convenience of the reader.

\begin{prop}
\label{prop:fratbasic}
Let $G$ be a finite or profinite group. Let $\Phi$ be the
Frattini subgroup of $G$. Let $H$ be a subgroup of $G$ that intersects all cosets of $\Phi$,
i.e., for which $\Phi H = G$. Then $H=G$.
\end{prop}

\begin{proof}
Suppose $H\neq G$.
Since $G$ is (pro)finite, there is a maximal (closed) subgroup $M\subsetneq G$ containing $H$.
Because $M\neq G$, there is some $\sigma\in G\smallsetminus M$.
We have $\Phi\subseteq M$ by definition of the Frattini subgroup.
However, since $\Phi H =G$, there exist some $\phi\in \Phi$ and $\tau\in H$
such that $\sigma = \phi \tau$. Therefore, since $\phi\in \Phi\subseteq M$
and $\tau\in H\subseteq M$, we have $\sigma\in M$.
The conclusion follows from this contradiction.
\end{proof}

We can use Theorem~\ref{thm:PinkEquiv} to compute the Frattini subgroup of $M_\infty$,
as follows.

\begin{thm}
\label{thm:fratM}
The Frattini subgroup $\Phi$ of the arithmetic basilica group $M_{\infty}$
has index $16$ in $M_{\infty}$, and it consists precisely of those $\sigma\in M_{\infty}$ for which:
\begin{itemize}
\item $\sigma$ fixes both nodes at level $1$ of the tree $T_{\infty}$,
\item $\sigma$ acts as an even permutation on the four nodes at level $2$ of the tree, and
\item $P(\sigma)\equiv 1 \pmod{8}$.
\end{itemize}
\end{thm}

\begin{proof}
\textbf{Step~1}. Let $K:=\QQ(t)$ 
and let $L:=K(\sqrt{t}, \sqrt{t+1}, \zeta_8)$. We claim that for any $a\in K$
for which $\sqrt{a}\in K_{\infty}$, we have $\sqrt{a}\in L$. To prove this claim, it
suffices to consider squarefree $a\in\ZZ[t]$ with $\sqrt{a}\in K_{\infty}$.

Let $K':=\Qbar(t)$, and consider any monic irreducible polynomial $h\in\Qbar[t]$ dividing $a$.
Then the prime $(h)$ of $K'$ ramifies in $K'(\sqrt{a})\subseteq K'_{\infty}$
and hence in $K'_n$ for some $n$, where $K'_n:=K'(f^{-n}(t))$ and $K'_{\infty}=\bigcup_n K_n$.
However, the discriminant $\Delta_n$ of the polynomial $f^n(z)-t$ is
\begin{equation}
\label{eq:Delta}
\Delta_n = (-4)^{2^{n-1}} \cdot \Delta_{n-1}^2 \cdot \big(f^n(0)-t \big) .
\end{equation}
(This is a standard iterative discriminant formula; see, for example,
\cite[Proposition~3.2]{AHM} or \cite[equation~(1)]{BenJuu}.)
Since $f^n(0)$ is always $0$ or $-1$,
it follows that the only finite primes at which $K'_n/K'$ ramifies are $(t)$ and $(t+1)$.
Therefore, $h=t$ or $h=t+1$.

Hence, $a$ is of the form $a=ct^i (t+1)^j$ where $c\in\ZZ$ is squarefree and $i,j\in\{0,1\}$.
If $c$ is divisible by an odd prime $p$, then working over $k:=\QQ$ by
specializing $t=1$, the prime $p$ ramifies in $k_n$ for some $n$,
where $k_n:=k(f^{-n}(1))$. However, for $t=1$, the discriminant $\Delta_n$
from equation~\eqref{eq:Delta} is of the form $(-1)^M 2^N$,
and in particular is not divisible by $p$, a contradiction.

Thus, $c$ must be $\pm 1$ or $\pm 2$.
Either way, we have $\sqrt{a}\in L$, proving our claim.

\medskip

\textbf{Step~2}.
Let $H$ be any maximal closed subgroup of $M_{\infty}$.
Since $M_{\infty}=\varprojlim M_n$ is a pro-2 group, we must have $[M_{\infty}:H]=2$.
The fields $k:=\QQ$ and $K:=\QQ(t)$ satisfy the hypotheses of Theorem~\ref{thm:PinkEquiv},
and hence with $x_0=t$, we have $G_{\infty}\cong M_{\infty}$.
The subfield $K_{\infty}^H$ of elements of $K_{\infty}$ fixed by the isomorphic image of
$H$ in $G_\infty$ is therefore a quadratic extension of $K$.
By the claim of Step~1, then, we have $K_{\infty}^H\subseteq L$,
where $L:=K(\sqrt{t}, \sqrt{t+1}, \zeta_8)$.
Thus, the Frattini subgroup of $G_{\infty}$ is contained in $\Psi:=\Gal(K_{\infty}/L)$.
 
Conversely, each of $\sqrt{t}$, $\sqrt{t+1}$, $\sqrt{2}$, $\sqrt{-1}$
generates a quadratic extension of $K$ lying in $K_{\infty}$,
and together they generate $L$.
That is, $\Gal(K_{\infty}/K(\sqrt{a}))$ is a maximal closed subgroup of $G_\infty$ for each of $a=t,t+1,2,-1$,
and the intersection of these four subgroups is $\Psi$.
Hence, the Frattini subgroup of $G_{\infty}$ is exactly $\Psi$,
which has index $[L:K]=16$ in $G_{\infty}$.
Therefore, the Frattini subgroup $\Phi$ of $M_{\infty}$ is the isomorphic copy of $\Psi$ in $M_\infty$,
with the same action on the tree.

By equation~\eqref{eq:Delta}, we have $\Delta_1=4(1+t)$ and $\Delta_2=-64(1+t)^2 t$,
so that $L=K(\sqrt{\Delta_1},\sqrt{\Delta_2},\zeta_8)$. Therefore, for any $\sigma\in G_{\infty}$,
we have $\sigma\in\Psi$ if and only if $\sigma$ fixes
$\sqrt{\Delta_1}$, $\sqrt{\Delta_2}$, and $\zeta_8$. That is, 
$\sigma\in\Psi$ if and only if $\sigma$ acts as an even permutation on both
the first and second levels of the tree, and $P(\sigma)\equiv 1\pmod{8}$,
yielding precisely the three desired bullet points.
(The third is by Theorem~\ref{thm:Pembed}.)
Therefore, the Frattini subgroup $\Phi$ of $M_{\infty}$ is the subset
carved out by the same conditions.
%
\end{proof}

\begin{remark}
Equation~\eqref{eq:Delta} shows that for $n\geq 2$, the discriminant $\Delta_n$
is always a square in $K$ times either $-t$ or $-t-1$, depending on whether $n$
is even or odd. (For $n=1$, $\Delta_1=4(1+t)$
is off by a sign from this pattern.)
Thus, for any $\sigma\in G_{\infty}$,
the parity of $\sigma$ as a permutation at each of the even levels of the tree is the same.
Similarly, $\sigma$ has the same parity at all odd levels of the tree from the third upwards;
the parity on the first level is the same as the other odd levels if and only if
$P(\sigma)\equiv 1 \pmod{4}$.

For $K=\QQ(t)$, Theorem~\ref{thm:PinkEquiv} shows that $G_{\infty}\cong M_{\infty}$,
and hence these parity observations carry over to the action of $M_{\infty}$.
In particular, for $\sigma\in M_{\infty}$, the three bullet points of Theorem~\ref{thm:fratM}
are equivalent to the following two conditions:
\begin{itemize}
\item $\sigma$ acts as an even permutation at every level of the tree, and
\item $P(\sigma)\equiv 1 \pmod{8}$.
\end{itemize}
\end{remark}

We are now prepared to prove statement~(2) of our Main Theorem,
which we state here in a more expanded form.

\begin{thm}
\label{thm:condition}
Fix notation as at the start of Section~\ref{sec:zeta}.
Fix roots of unity $\zeta_{2^m}$ and a tree labeling as in Lemma~\ref{lem:pickzeta}.
The following are equivalent.
\begin{enumerate}
\item $[K(\sqrt{-x_0}, \sqrt{1+x_0},\zeta_8) : K]=16$.
\item $[K_5:K]=2^{25}$.
\item $G_5\cong M_5$.
\item $G_n\cong M_n$ for all $n\geq 1$.
\item $G_{\infty}\cong M_{\infty}$.
\end{enumerate}
As always, the isomorphisms of statements~(3)--(5) are of groups acting on trees,
not just of abstract groups.
\end{thm}

\begin{proof}
\textbf{Step 1}.
The implications (5)$\Rightarrow$(4)$\Rightarrow$(3) are trivial.
By Corollary~\ref{cor:Pbasil}, $G_5$ is isomorphic to a subgroup of $M_5$,
and $|M_5|=2^{25}$ by Theorem~\ref{thm:MBmap}.
Therefore, since $|G_5|=[K_5:K]$, we have
(2)$\Leftrightarrow$(3).

\medskip

\textbf{Step 2}.
To prove (1)$\Rightarrow$(5), define
\[ L:=K(\sqrt{-x_0}, \sqrt{1+x_0},\zeta_8)
= K(\sqrt{\Delta_1},\sqrt{\Delta_2},\zeta_8), \]
where $\Delta_1:= 4(1+x_0)$ and $\Delta_2:=-64(1+x_0)^2 x_0$,
which are the discriminants of the polynomials $f(z)-x_0$ and $f^2(z)-x_0$,
by equation~\eqref{eq:Delta}. Thus, by Lemma~\ref{lem:nrel},
we have $L\subseteq K_{\infty}$.

Since $|\Gal(L/K)|=16$, each of the 16 combinations of
\[ \Delta_1 \mapsto \pm \Delta_1, \quad \Delta_2 \mapsto \pm \Delta_2, \quad
\zeta_8\mapsto \pm \zeta_8, \pm \zeta_8^3 \]
is realized by some $\tau\in\Gal(L/K)$. Lifting each such $\tau$ to $\sigma\in G_{\infty}$,
it follows that each of the 16 combinations of
\begin{itemize}
\item $\sigma$ is even or odd at level $1$ of the tree,
\item $\sigma$ is even or odd at level $2$ of the tree,
\item $P(\sigma)$ is $1,3,5,7\pmod{8}$
\end{itemize}
is realized by some $\sigma\in G_{\infty}$.

By Corollary~\ref{cor:Pbasil}, $G_{\infty}$ is isomorphic to a subgroup $H$ of $M_\infty$.
By the previous paragraph, $H$ intersects all the cosets
of the Frattini subgroup $\Phi$ of $M_{\infty}$.
Therefore, by Theorem~\ref{thm:fratM}, the subgroup $H$ is all of $M_{\infty}$,
and hence $G_{\infty}\cong M_{\infty}$.

\medskip

\textbf{Step 3}. It remains to show (3)$\Rightarrow$(1).
Use the same notation $L$, $\Delta_1$, and $\Delta_2$ from Step~2.
Since $\zeta_8\in K_5$ by Lemma~\ref{lem:nrel},
and since $\sqrt{\Delta_1}\in K_1$ and $\sqrt{\Delta_2}\in K_2$,
we have $L\subseteq K_5$.

Let $\alpha, \beta, \eps, \theta\in G_5$ denote the images
of $\mapa, \mapb,\mape,\mapt\in M_{\infty}$ from Section~\ref{sec:Basil}
under the homomorphism $R_{\infty,5}:M_{\infty}\to M_5$
composed with the isomorphism $M_5\cong G_5$.

Then $\alpha$, like $\mapa$, switches the two nodes at level $1$ of the tree
but is even on level~2, and moreover satisfies $P(\alpha)\equiv 1 \pmod{8}$.
Similarly, $\beta$ fixes the two nodes at level $1$ of the tree
but is odd on level~2, and again satisfies $P(\alpha)\equiv 1 \pmod{8}$.
On the other hand, $\eps$ switches the two nodes at level~1 but is even at level~2,
this time with $P(\eps)\equiv -1 \pmod{8}$.
Therefore, $\alpha \eps$ is even at both levels~1 and~2 and satisfies
$P(\alpha \eps)\equiv -1 \pmod{8}$.
Finally, $\theta$ is even at both levels~1 and~2 and satisfies
$P(\theta)\equiv 3 \pmod{8}$.

Hence, $\alpha$ fixes $\sqrt{\Delta_2}$ and $\zeta_8$
but maps $\sqrt{\Delta_1}$ to its negative.
Similarly, $\beta$ fixes $\sqrt{\Delta_1}$ and $\zeta_8$
but maps $\sqrt{\Delta_2}$ to its negative.
Both $\alpha\eps$ and $\theta$ fix both $\sqrt{\Delta_1}$ and $\sqrt{\Delta_2}$
but map $\zeta_8$ to $\zeta_8^{-1}$ and $\zeta_8^3$, respectively.
Restricting each to $L$, then, we see that $\Gal(L/K)$ has at least $16$ elements,
and hence exactly $16$. That is, $[L:K]=16$.
\end{proof}

\textbf{Acknowledgements}.
The results of this paper grew out of an REU project at Amherst College.
Several of our results were originally suggested by computations using Magma.
Authors RB, JC, and GC gratefully acknowledge the support of NSF grant DMS-1501766.
Authors FA and LF gratefully acknowledge the support of Amherst College's
Gregory S.~Call student research funding.
We thank Harris Daniels and Jamie Juul for helpful discussions,
and Joseph Lupo for spotting some minor errors in the original manuscript.
A particular thanks goes to Rafe Jones, for proposing this problem,
for providing useful background,
and for suggestions surrounding Conjectures~\ref{conj:basic} and~\ref{conj:general}.
Finally, we thank the referee for their careful reading of the paper and their
apt suggestions for improvements, including a significant simplification of the
proof of Theorem~\ref{thm:condition}, by leveraging the results of \cite[Section~2.8]{Pink}
via the Frattini subgroup strategy of Theorem~\ref{thm:fratM}.

\bibliographystyle{amsalpha}

\begin{thebibliography}{99}
\bibitem{AHM}
Wayne Aitken, Farshid Hajir, and Christian Maire,
Finitely ramified iterated extensions,
\emph{Int. Math. Res. Not.} \textbf{2005}, 855--880.

\bibitem{ABetal}
Jacqueline Anderson, Irene I.\ Bouw, Ozlem Ejder, Neslihan Girgin, Valentijn Karemaker, and Michelle Manes,
Dynamical Belyi maps,
in \emph{Women in numbers Europe II},
Springer, Cham (2018), 57--82.


\bibitem{BGN}
Laurent Bartholdi, Rostislav Grigorchuk, and Volodymyr Nekrashevych, 
From fractal groups to fractal sets,
in \emph{Fractals in Graz 2001}, 25--118, 
Birkh\"{a}user, Basel, 2003. 


\bibitem{BFHJY}
Robert L.\ Benedetto, Xander Faber, Benjamin Hutz, Jamie Juul, and Yu Yasufuku,
A large arboreal {G}alois representation for a cubic postcritically finite polynomial,
\emph{Res.\ Number Theory} \textbf{3} (2017), Art.\ 29, 21.

\bibitem{BenJuu}
Robert L.\ Benedetto and Jamie Juul,
Odoni's conjecture for number fields,
\emph{Bull.\ Lond.\ Math.\ Soc.}\ \textbf{51} (2019), 237--350.

\bibitem{BosJon}
Nigel Boston and Rafe Jones,
Arboreal {G}alois representations,
\emph{Geom. Dedicata} \textbf{124} (2007), 27--35.


\bibitem{BriTuc}
Andrew Bridy and Thomas J.\ Tucker,
\emph{Finite index theorems for iterated {G}alois groups of cubic polynomials},
\emph{Math.\ Ann.}\ \textbf{373} (2019), 37--72.


\bibitem{BHL}
Michael R.\ Bush, Wade Hindes, and Nicole R.\ Looper,
Galois groups of iterates of some unicritical polynomials,
\emph{Acta Arith.}\ \textbf{181} (2017), 57--73.

\bibitem{FM}
Andrea Ferraguti and Giacomo Micheli,
An equivariant isomorphism theorem for mod~$\mathfrak{p}$
reductions of arboreal Galois representations,
preprint, 2019.
Available at \texttt{arXiv:1905.00506}.

\bibitem{FPC}
Andrea Ferraguti, Carlo Pagano, and Daniele Casazza,
The inverse problem for arboreal Galois representations of index two,
preprint, 2019.
Available at \texttt{arXiv:1907.08608}.

\bibitem{GoTa}
Richard Gottesman and Kwokfung Tang,
Quadratic recurrences with a positive density of prime divisors,
\emph{Int.\ J.\ Number Theory}, \textbf{6} (2010), 1027--1045.

\bibitem{GNT}
Chad Gratton, Khoa Nguyen, and Thomas J.\ Tucker,
$ABC$ implies primitive prime divisors in arithmetic dynamics,
\emph{Bull.\ Lond.\ Math.\ Soc.}\ \textbf{45} (2013), 1194--1208.

\bibitem{Hindes}
Wade Hindes,
Average Zsigmondy sets, dynamical Galois groups, and the Kodaira-Spencer map,
\emph{Trans.\ Amer.\ Math.\ Soc.}\ \textbf{370} (2018), 6391--6410.

\bibitem{Hindes2}
Wade Hindes,
Classifying Galois groups of small iterates via rational points,
\emph{Int.\ J.\ Number Theory} \textbf{14} (2018), 1403--1426. 

\bibitem{Ing}
Patrick Ingram,
Arboreal Galois representations and uniformization of polynomial dynamics,
\emph{Bull.\ Lond.\ Math.\ Soc.}\ \textbf{45} (2013), 301--308. 

\bibitem{Jones}
Rafe Jones,
Galois representations from pre-image trees: an arboreal survey,
in \emph{Actes de la Conf\'{e}rence ``Th\'{e}orie des Nombres et Applications''},
\emph{Pub.\ Math.\ Besan\c{c}on} (2013), 107-136.

\bibitem{JonMan}
Rafe Jones and Michelle Manes,
Galois theory of quadratic rational functions,
\emph{Comment.\ Math.\ Helv.}\ \textbf{89} (2014), 173--213.


\bibitem{Juul}
Jamie Juul,
Iterates of generic polynomials and generic rational functions,
\emph{Trans.\ Amer.\ Math.\ Soc.}\ \textbf{371} (2019), 809--831.


\bibitem{JKetal}
Jamie Juul, Holly Krieger, Nicole Looper, Michelle Manes, Bianca Thompson, and Laura Walton,
Arboreal representations for rational maps with few critical points,
preprint, 2018.
Available at \texttt{arXiv:1804.06053}.


\bibitem{JKMT}
Jamie Juul, P\"{a}r Kurlberg, Kalyani Madhu, and Tom J.\ Tucker, 
Wreath products and proportions of periodic points,
\emph{Int.\ Math.\ Res.\ Not.\ IMRN} \textbf{2016}, 3944--3969.
 

\bibitem{Kadets}
Borys Kadets,
Large arboreal Galois representations,
preprint, 2018.
Available at \texttt{arXiv:1802.09074}.

\bibitem{Looper}
Nicole Looper,
Dynamical Galois groups of trinomials and Odoni's conjecture,
\emph{Bull.\ Lond.\ Math.\ Soc.}\  \textbf{51} (2019), 278--292.


\bibitem{Nek}
Volodymyr Nekrashevych,
\emph{Self-Similar Groups},
American Mathematical Society, Providence, 2005.

\bibitem{Odoni}
R.\ W.\ K.\ Odoni,
The Galois theory of iterates and composites of polynomials,
\emph{Proc.\ London Math.\ Soc.\ (3)} \textbf{51} (1985), no. 3, 385--414.

\bibitem{Pink}
Richard Pink,
Profinite iterated monodromy groups arising from quadratic polynomials,
preprint 2013.
Available at \texttt{arXiv:1307.5678}. 


\bibitem{Serre}
Jean-Pierre Serre,
Propri\'{e}t\'{e}s galoisiennes des points d’ordre fini des courbes elliptiques,
\emph{Invent.\ Math.}\ \textbf{15} (1972), 259--331.

\bibitem{Specter}
Joel Specter,
Polynomials with Surjective Arboreal Galois Representations Exist in Every Degree,
preprint, 2018.
Available at \texttt{arXiv:1803.00434}.

\bibitem{Stoll}
Michael Stoll,
Galois groups over \textbf{Q} of some iterated polynomials,
\emph{Arch.\ Math.\ (Basel)} \textbf{59} (1992), 239--244.

\bibitem{Swam}
Ashvin A.\ Swaminathan,
On arboreal Galois representations of rational functions,
\emph{J.\ Algebra} \textbf{448} (2016), 104--126.


\end{thebibliography}

\end{document}